[1994/12/01]
\documentclass{amsart}
\usepackage{mathptmx}

\usepackage[english]{babel}
\usepackage[utf8]{inputenc}
\usepackage{textcomp}  
\usepackage[T1]{fontenc}
\usepackage{graphicx}

\usepackage[automark]{scrpage2}
\usepackage[a4paper,left=4cm, right=2cm,top=3cm, bottom=3cm]{geometry}

\usepackage{amsthm,amsmath,amssymb,amsfonts}
\usepackage{mathrsfs}

\newtheorem{Theorem}{Theorem}[section]
\newtheorem{Lemma}[Theorem]{Lemma}
\newtheorem{Satz}[Theorem]{Proposition}
\newtheorem{Folgerung}[Theorem]{Corollary}

\theoremstyle{definition} 

\newtheorem{Bemerkung}[Theorem]{Remark}
\newtheorem{Definition}[Theorem]{Definition}

\newenvironment{bew}{\begin{proof}[Proof]}{\end{proof}}

\newcommand{\cal}{\mathcal}
\newcommand{\re}{\mathbb{R}^n}
\newcommand{\R}{\mathbb{R}}
\newcommand{\C}{\mathbb{C}}
\newcommand{\N}{\mathbb{N}}
\newcommand{\Z}{\mathbb{Z}}
\newcommand{\bspq}{B^s_{p,q}(\re) }
\newcommand{\fspq}{F^{s}_{p,q}(\re)}
\newcommand{\sint}{\lfloor s\rfloor}
\newcommand{\srest}{\{s\}}

\newcommand{\Lint}{\lfloor L\rfloor}
\newcommand{\Lrest}{\{L\}}
\newcommand{\rint}{\lfloor \rho\rfloor}

\newcommand{\lip}[1][\sigma]{lip^{#1}(\re)}
\newcommand{\Kap}[1][L]{\varkappa_{#1}}
\newcommand{\hold}[1][s]{{\cal C}^{#1}(\re)}

\begin{document}


\title{Atomic representations in function spaces and applications to pointwise multipliers and diffeomorphisms, a new approach}

\author[Benjamin Scharf]{Benjamin Scharf} 
\email{benjamin.scharf@uni-jena.de} 
\address{Mathematisches Institut\\ Fakult\"at f\"ur Mathematik und Informatik\\ Friedrich-Schiller-Universit\"at Jena\\ D-07737 Jena\\ Germany}
\thanks{Benjamin Scharf: Mathematisches Institut, Fakult\"at f\"ur Mathematik und Informatik, Friedrich-Schiller-Universit\"at Jena, D-07737 Jena, Germany, benjamin.scharf@uni-jena.de}



\begin{abstract}
In Chapter 4 of \cite{Tri92} Triebel proved two theorems concerning pointwise multipliers and diffeomorphisms in function spaces $\bspq$ and $\fspq$. In each case he presented two approaches, one via atoms and one via local means. While the approach via atoms was very satisfactory concerning the length and simplicity, only the rather technical approach via local means proved the theorems in full generality.

In this paper we generalize two extensions of these atomic decompositions, one by Skrzypczak (see \cite{Skr98}) and one by Triebel and Winkelvoss (see \cite{TrW96}) so that we are able to give a short proof using atomic representations getting an even more general result than in the two theorems in \cite{Tri92}.    
\end{abstract}

\maketitle


\section*{Introduction}

The aim of this paper is to generalize the atomic decomposition theorem from Triebel \cite{Tri92, Tri97} for Besov and Triebel-Lizorkin spaces $\bspq$ and $\fspq$ and to present two applications to pointwise multipliers and diffeomorphisms as continuous linear operators in $\bspq$ resp. $\fspq$. For a detailed (historical) treatment of the spaces $\bspq$ and $\fspq$ we refer to Triebel \cite{Tri83, Tri92}, for an introduction to atoms we refer to Frazier and Jawerth \cite{FrJ85,FrJ90}. 

According to Triebel \cite{Tri92}
 \begin{align*}
 P_{\varphi}: f \mapsto \varphi \cdot f
\end{align*}
maps $\bspq$ into $\bspq$ if $s> \sigma_p$ and $\varphi \in C^k(\re)$ with $k>s$. Furthermore, the superposition with a vector function $\varphi: \re \rightarrow \re$
 \begin{align*}
 D_{\varphi}: f \mapsto f \circ \varphi
\end{align*}
maps $\bspq$ to $\bspq$ if $\varphi$ is a $k$-diffeomorphism and $k$ is large enough in dependence of $s$ and $p$. There are similar results for $\fspq$. 

The main idea for an easy proof is the atomic decomposition theorem. Mainly one has to show that a multiplication of an atom $a_{\nu,m}$ with a function $\varphi$ resp. the superposition with $\varphi$ is still an atom with similar properties. But there was one problem: If $s \leq \sigma_p$ resp. $s \leq \sigma_{p,q}$, then atoms need to fulfil moment conditions, i.e.
\begin{align}
\label{momento}
\int_{\mathbb{R}^n} &x^{\beta} a(x) \ dx=0 \text{ if }  |\beta|\leq L-1
\end{align}
for $L \in \N_0$ and $L>\sigma_p-s$ resp. $L>\sigma_{p,q}-s$. But these properties are not preserved by multiplication resp. superposition. By Skrzypczak \cite{Skr98} these moment conditions were replaced by the more general assumptions
\begin{align*} 
\left| \int_{d \cdot Q_{\nu,m}} \psi(x) a(x) \ dx \right|\leq C \cdot  2^{-\nu\left(s+L+n\left(1-\frac{1}{p}\right)\right)} \|\psi|C^L(\re)\|
\end{align*}
for all $\psi \in C^L(\re)$. Now the situation changes: These conditions remain true after multiplication resp. superposition.

This replacement is typical when thinking of atomic, in particular wavelet representations as representations of functions not mapping from $\re$, but from more general manifolds, see the remarks on the cancellation property in \cite[Section 3.1]{Dah01}. 

In this paper we go a step further. We show that one can replace the usual $C^K(\re)$-conditions on atoms by Hölder-conditions ($\hold[K]$-spaces) in the following way:

\textit{ A function $a:\re \rightarrow \C$ is called $(s,p)_{K,L}$-atom located at $Q_{\nu,m}$ if 
\begin{align*}
 supp \ a &\subset d \cdot Q_{\nu,m}  \\
  \|a(2^{-\nu}\cdot)|\hold[K]\| &\leq C \cdot 2^{-\nu(s-\frac{n}{p})}
\end{align*}
and for every $\psi \in \hold[L]$ it holds 	
\begin{align*}
 \left| \int_{d \cdot Q_{\nu,m}} \psi(x) a(x) \ dx \right|\leq C \cdot  2^{-\nu\left(s+L+n\left(1-\frac{1}{p}\right)\right)} \|\psi|\hold[L]\|.
\end{align*}}
This generalizes the known definitions of atoms from Triebel, Skrzypczak and Winkelvoss \cite{Tri97, Skr98, TrW96}. 

Furthermore, there is an existing theory generalizing the conditions $\|a(2^{-\nu}\cdot)|\hold[K]\|$ by $\|a(\cdot)|B_{p,p}^{K}(\re)\|$ with $K>s$, mainly in connection with spline representations. For instance, see the books by Kahane and Lemarie-Rieusset \cite[part II, Section 6.5]{KaL95}, Triebel \cite[Section 2.2]{Tri06} and the recent paper by Schneider and Vybiral \cite{ScV12}. Of these, only the first book incorporates the usual moment conditions as in \eqref{momento}.

\medskip
In Section 3, as corollaries of the atomic representation theorem with these more general atoms from Section 2 we are able to extend the key theorems on pointwise multipliers and diffeomorphisms from \cite{Tri92}.
It is not the aim of our observations to give best conditions or even exact characterizations for pointwise multipliers in function spaces $\bspq$ and $\fspq$. For this we refer to Strichartz \cite{Str67}, Peetre \cite{Pee76} as well as to Maz'ya and Shaposhnikova \cite{MaS85, MaS09} for the classical Sobolev spaces, while for $B_{p,p}^s(\re)$ we refer to Franke \cite{Fra86}, Frazier and Jawerth \cite{FrJ90}, Netrusov \cite{Net92}, Koch, Runst and Sickel \cite{RS96, Sic99a, Sic99b, KoS02} as well as to Triebel \cite[Section 2.3.3]{Tri06} for general function spaces $\bspq$ and $\fspq$.

We obtain for pointwise multipliers with respect to $\bspq$:

\textit{Let $0<p\leq \infty$ and $\rho > \max(s,\sigma_p-s)$. Then there exists a positive number $c$ such that
\begin{align*}
 \|\varphi f|\bspq\| \leq c \|\varphi|\hold[\rho]\| \cdot \|f|\bspq\|
\end{align*}
for all $\varphi \in \hold[\rho]$ and all $f \in \bspq$.
}

\noindent

For further sufficient results on diffeomorphisms including characterizations for classical Sobolev spaces $W_{p}^k(\re)$ we refer to Gol'dshtein, Reshetnyak, Romanov, Ukhlov and Vodop'yanov \cite{GoR76, GoV75, GoV76, Vod89, UkV02},  \cite[Chapter 4]{GoR90}, Markina \cite{Mar90} as well as to Maz'ya and Shaposhnikova \cite{Maz69, MaS85}, while for Besov spaces $\bspq$ with $0<s<1$ we refer to Vodop'yanov, Bordaud and Sickel \cite{Vod89, BoS99}. A special case of our result (Lipschitz diffeomorphisms) can be found in Triebel \cite[Section 4]{Tri02}.

We will prove (in case of $\bspq$):

\textit{Let $0<p\leq \infty$, $\rho\geq 1$ and $\rho > \max(s,1+\sigma_p-s)$. If $\varphi$ is a $\rho$-diffeomorphism, then there exists a constant $c$ such that
\begin{align*}
 \|f(\varphi(\cdot))|\bspq\| \leq c \cdot \|f|\bspq\|.
\end{align*}
for all $f \in \bspq$. Hence $D_{\varphi}$ maps $\bspq$ onto $\bspq$.}

Furthermore, at the end of section 2 we are able to give a simple proof of a local mean theorem very similar to Triebel's result \cite[Theorem 1.15]{Tri08}, which paved the way for the wavelet characterization of $\bspq$ and $\fspq$ - where we are also using the more general Hölder-space conditions.

\section{Preliminaries}
Let $\re$ be the euclidean $n$-space, $\Z$ be the set of integers,$\N$ be the set of natural numbers and $\N_0=\N \cup \{0\}$. By $|x|$ we denote the usual euclidean norm of $x \in \re$, by $\|x|X\|$ the (quasi)-norm of an element $x$ of a (quasi)-Banach space $X$.  

By ${\cal S}(\re)$ we mean the Schwartz space on $\re$, by ${\cal S}'(\re)$ its dual. The Fourier transform of $f \in {\cal S}'(\re)$ resp. its inverse will be denoted by $\hat{f}$ resp. $\check{f}$. The convolution of $f \in {\cal S}'(\re)$ and $\varphi \in {\cal S}(\re)$ will be denoted by $f * \varphi$.

By $L_p(\re)$ for $0<p\leq \infty$ we denote the usual quasi-Banach space of $p$-integrable complex-valued functions with respect to the Lebesgue measure $\mu$ with quasi-norm
\begin{align*}
 \|f|L_p(\re)\|:=\left(\int_{\re} |f(x)|^p \ dx\right)^{\frac{1}{p}}.
\end{align*}

Let $X,Y$ be quasi-Banach spaces. By the notation $X \hookrightarrow Y$ we mean that $X \subset Y$ and that the inclusion map is bounded. 

Throughout the paper all unimportant constants will be called $c,c',C$ etc. Only if extra clarity is desirable, the dependency of the parameters will be stated explicitly. The concrete value of these constants may vary in different formulas but remains the same within one chain of inequalities.
 \subsection{H\"older spaces of differentiable functions}
Let $k \in \N_0$. Then by $C^k(\re)$ we denote the space of all functions $f: \re \rightarrow \C$ which are $k$-times continuously differentiable (continuous, if $k=0$) such that the norm
\begin{align*}
 \|f|C^k(\re)\|:=\sum_{|\alpha|\leq k} \sup |D^{\alpha} f(x)|
\end{align*}
is finite, where the $\sup$ is taken over $x \in \re$.

Furthermore, the set $C^{\infty}(\re)$ is defined by
\begin{align*}
 C^{\infty}(\re):=\bigcap_{k \in \N_0} C^k(\re).
\end{align*}

\begin{Definition}
\label{Hoelder}
 Let $0<\sigma\leq 1$ and $f: \re \rightarrow \C$ be continuous. We define
\begin{align*}
 \|f|\lip[\sigma]\|:=\sup_{x,y\in \re, x\neq y} \frac{|f(x)-f(y)|}{|x-y|^{\sigma}}.
\end{align*}

If $s\in \mathbb{R}$, then there are uniquely determined $\sint \in \mathbb{Z}$ and $\srest \in (0,1]$ with $s=\sint+\srest$. 

Let $s>0$. Then the H\"older space with index $s$ is given by
\begin{align*}
 \hold&=\left\{f \in C^{\sint}(\re):  \|f|{\hold}\|<\infty\right\} \text{ with } \\
 \|f&|{\hold}\|:=\|f|C^{\sint}(\re)\|+\sum_{|\alpha|=\sint} \|D^{\alpha} f|\lip[\srest]\|.
\end{align*}
If $s=0$, then $\hold[0]:=L_{\infty}(\re)$, which is sufficient for the later statements, see e.g. Theorem \ref{AtomicRepr}.

\end{Definition}

\subsection{Besov and Triebel-Lizorkin function spaces on $\re$}
Let $\varphi_j$ for $j \in \mathbb{N}_0$ be elements of ${\cal S}(\mathbb{R}^n)$ with
\begin{align}
 \label{GrundResolution}
 \begin{split}
 supp \ \varphi_0 &\subset  \{|\xi| \leq 2\}, \\
 supp \ \varphi_j &\subset\{2^{j-1}\leq |\xi| \leq 2^{j+1}\} \text{ for } j \in \mathbb{N}, \\
 \sum_{j=0}^{\infty} \varphi_j(\xi)&=1 \text{ for all } \xi \in \mathbb{R}^n, \\
 |D^{\alpha} \varphi_j(\xi)| &\leq c_{\alpha} 2^{-j|\alpha|} \text{ for all } \alpha \in \mathbb{N}_0^n.
\end{split}
\end{align}
Then we call $\{\varphi_j\}_{j=0}^{\infty}$ a smooth dyadic resolution of unity. For instance one can choose $\Psi \in {\cal S}(\mathbb{R}^n)$ with $\Psi(\xi)=1$ for $|\xi|\leq 1$ and $supp \ \Psi \subset \{|\xi|\leq 2 \}$ and set 
\begin{align*}
 \varphi_0(\xi):=\Psi(\xi), \quad \varphi_1(\xi):=\Psi(\xi/2)-\Psi(\xi), \quad \varphi_j(\xi):=\varphi_1(2^{-j+1}\xi) \text{ for } j \in \mathbb{N}.
\end{align*}
\begin{Definition}
 \label{GrundDefinitionB}
Let $0<p\leq \infty$, $0<q\leq \infty$, $s \in \mathbb{R}$ and $\{\varphi_j\}_{j=0}^{\infty}$ be a smooth dyadic resolution of unity. Then $\bspq$ is the collection of all  $f\in {\cal S}'(\mathbb{R}^n)$ such that the quasi-norm
\begin{align*}
 \|f|\bspq\|:=\left(\sum_{j=0}^{\infty} 2^{jsq}\|(\varphi_j \hat{f})\check{\ }|L_p(\re)\|^q\right)^{\frac{1}{q}}
\end{align*}
(modified if $q=\infty$) is finite.
\end{Definition}
\begin{Definition}
 \label{GrundDefinitionF}
Let $0<p<\infty$, $0<q\leq \infty$, $s \in \mathbb{R}$ and $\{\varphi_j\}_{j=0}^{\infty}$ be a smooth dyadic resolution of unity. Then $\fspq$ is the collection of all $f\in {\cal S}'(\mathbb{R}^n)$ such that the quasi-norm
\begin{align*}
 \|f|\fspq\|:=\left\|\left(\sum_{j=0}^{\infty} 2^{jsq} \left|(\varphi_j \hat{f})\check{\ }(\cdot)\right|^q \right)^{\frac{1}{q}}\big|L_p(\re)\right\|\end{align*}
(modified if $q=\infty$) is finite.
\end{Definition}
One can show that the introduced quasi-norms\footnote{In the following we will use the term ``norm`` even if we only have quasi-norms for $p<1$ or $q<1$.} for two different smooth dyadic resolutions of unity are equivalent for fixed $p$, $q$ and $s$, i. e. that the so defined spaces are equal. This follows from Fourier multiplier theorems, see \cite{Tri83}, Section 2.3.2., p. 46. Furthermore, the so defined spaces are (quasi)-Banach spaces.

The next proposition, the so called Fatou property, is a classical observation for function spaces $\bspq$ and $\fspq$, see \cite{Fra86}.
\begin{Definition}
 Let $A$ be a quasi-Banach space with ${\cal S}(\re) \hookrightarrow A \hookrightarrow {\cal S}'(\re)$. Then we say that $A$ has the Fatou property if there exists a constant $c$ such that:
 If a sequence $\{f_n\}_{n \in \N} \subset A$ converges to $f$ with respect to the weak topology in ${\cal S}'(\re)$ and if $\|f_n|A\|\leq D$, then $f \in A$ and $\|f|A\|\leq c \cdot D$.
\end{Definition}

\begin{Satz}
\label{Fatou}
Let $s \in \R, 0<q\leq \infty$ and $0<p\leq \infty$ resp. $0<p<\infty$. Then $\bspq$ and $\fspq$ have the Fatou property.
\end{Satz}

\begin{Bemerkung}
\label{HoldB}
 If $\rho>0$ and $\rho \notin \N$, then $\hold[\rho]=B_{\infty,\infty}^{\rho}(\re)$. This is a classical observation, for instance see \cite[Sections 1.2.2, 2.6.5]{Tri92} or for the original source \cite[Lemma 4]{Zyg45}.
\end{Bemerkung}

Moreover, we set
\begin{align*}
\sigma_{p}=n\left(\frac{1}{p}-1\right)_+, \quad \sigma_{p,q}=n\left(\frac{1}{\min(p,q)}-1\right)_+,
\end{align*}
where $a_+=\max(a,0)$. 

\section{Atomic decompositions}
At first we describe the concept of atoms as one can find it in \cite{Tri97}, Definition 13.3, p. 73, now generalized using ideas from \cite{Skr98} and \cite{TrW96}. 

In particular, this gives the possibility to omit the distinction between $\nu=0$ and $\nu \in \mathbb{N}$ and now the usual parameters $K$ and $L$ are nonnegative real numbers instead of natural numbers.

\subsection{General atoms}
Let $Q_{\nu,m}:=\{x \in \mathbb{R}^n: |x_i-2^{-\nu}m_i|\leq 2^{-\nu-1}\}$ be the cube with sides parallel to the axes, with center at $2^{-\nu}m$ and side length $2^{-\nu}$ for $m \in \mathbb{Z}^n$ and $\nu \in \mathbb{N}_0$.
\begin{Definition}
\label{Atoms}
Let $s \in \mathbb{R}$, $0<p\leq\infty$, $K,L \in \R$ and $K,L\geq 0$. Furthermore let $d>1$, $C>0$, $\nu \in \mathbb{N}_0, m \in \mathbb{Z}$. A function $a:\re \rightarrow \mathbb{C}$ is called $(s,p)_{K,L}$-atom located at $Q_{\nu,m}$ if 
\begin{align}
 \label{Atom1}supp \ a &\subset d \cdot Q_{\nu,m}  \\
 \label{Atom2} \|a(2^{-\nu}\cdot)|\hold[K]\| &\leq C \cdot 2^{-\nu(s-\frac{n}{p})}
\end{align}
and for every $\psi \in \hold[L]$ it holds 	
\begin{align}
\label{Atom3} \left| \int_{d \cdot Q_{\nu,m}} \psi(x) a(x) \ dx \right|\leq C \cdot  2^{-\nu\left(s+L+n\left(1-\frac{1}{p}\right)\right)} \|\psi|\hold[L]\|.
\end{align}
The constant in the exponent will be shortened by $\Kap:=s+L+n\left(1-\frac{1}{p}\right)$. 
\end{Definition}
\begin{Bemerkung}
If $L=0$, then condition \eqref{Atom3} is neglectable since it follows from $\eqref{Atom1}$ and $\eqref{Atom2}$ with $K=0$.

If $K=0$, then by Definition \ref{Hoelder} we only require $a$ to be suitably bounded.

Later on, we will choose one $(s,p)_{K,L}$-atom for every $\nu \in \N_0$ and $m \in \Z^n$. Then the parameter $d>1$ shall be the same for all these atoms - it describes the overlap of these atoms on one fixed level $\nu \in \N_0$.
\end{Bemerkung}

\begin{Bemerkung}
\label{usualform}
The usual formulation of $\eqref{Atom2}$ as in \cite{Tri97} was
\begin{align}
\label{Atom21}
  |D^{\alpha} a(x)| &\leq 2^{-\nu\left(s-\frac{n}{p}\right)+|\alpha|\nu} \text{ for all } |\alpha|\leq K
\end{align}
for $K \in \N_0$. The modification here was suggested in \cite{TrW96}. It is easy to see that \eqref{Atom2} follows from \eqref{Atom21} if $K$ is a natural number, since $C^K(\re) \hookrightarrow \hold[K]$.  
\end{Bemerkung}
\begin{Bemerkung}
\label{Ordered}
The usual formulation of $\eqref{Atom3}$ as in \cite{Tri97} was
\begin{align}
  \label{Atom31}\int_{\mathbb{R}^n} &x^{\beta} a(x) \ dx=0 \text{ if }  |\beta|\leq L-1
\end{align}
for $\nu \in \mathbb{N}$, so $\nu \neq 0$. The modification here was suggested in Lemma 1 of \cite{Skr98} for natural numbers $L+1$ (using $C^{L}(\re)$ instead of $\hold[L]$). Now we extended this definition to general positive $L$. For natural $L-1$ one can derive $\eqref{Atom3}$ from $\eqref{Atom31}$ using a Taylor expansion, see \cite[Lemma 1, (12) and (14)]{Skr98} or the upcoming Lemma \ref{Momenter}. Hence formulation $\eqref{Atom3}$ is a generalization. 

An alternative formulation of \eqref{Atom3} is given by
\begin{align}
 \label{Atom32}
\left| \int_{d \cdot Q_{\nu,m}} (x-2^{-\nu}m)^{\beta} a(x) \ dx \right|\leq C \cdot  2^{-\nu\Kap} \text{ if } |\beta|\leq \Lint.
\end{align}
Obviously, this condition is covered by condition \eqref{Atom3}. For the other direction see \cite[Lemma 1, (12) and (14)]{Skr98} or the upcoming Remark \ref{MomentBem}, in particular \eqref{IntMomenter}. It is also possible to assume this condition for all $\beta\in \N^n$ since the statements for $|\beta|\geq L$ follow from the support condition \eqref{Atom1} and the boundedness condition included in \eqref{Atom2}.

This shows that both conditions \eqref{Atom2} and \eqref{Atom3} are ordered in $K$ resp. $L$, i.e. the conditions get stricter for increasing $K$ resp. $L$.
\end{Bemerkung}

Now the question will be whether these more general atoms allow analogous results regarding atomic decompositions.

\subsection{Sequence spaces}
We introduce the sequence spaces $b_{p,q}$ and $f_{p,q}$, whose use will become clear in the following. For this we refer to \cite{Tri97}, Definition 13.5, p. 74.
\begin{Definition}
\label{DefSeq}
 Let $0<p\leq \infty$, $0<q\leq \infty$ and 
 \begin{align*}
  \lambda=\left\{ \lambda_{\nu,m} \in \mathbb{C}: \nu \in \mathbb{N}_0, m \in \mathbb{Z}^n\right\}.
 \end{align*}
We set
\begin{align*}
 b_{p,q}:=\left\{\lambda: \|\lambda|b_{p,q}\|=\left(\sum_{\nu=0}^{\infty} \left(\sum_{m \in \mathbb{Z}^n} |\lambda_{\nu,m}|^p\right)^{\frac{q}{p}} \right)^{\frac{1}{q}} <\infty \right\}
\end{align*}
and
\begin{align*}
 f_{p,q}:=\left\{\lambda: \|\lambda|f_{p,q}\|=\left\|\left(\sum_{\nu=0}^{\infty} \sum_{m \in \mathbb{Z}^n} |\lambda_{\nu,m}\chi_{\nu,m}^{(p)}(\cdot)|^q\right)^{\frac{1}{q}}\big|L_p(\re)\right\| <\infty \right\}
\end{align*}
(modified in the case $p=\infty$ or $q=\infty$), where $\chi_{\nu,m}^{(p)}$ is the $L_p(\re)$-normalized characteristic function of the cube $Q_{\nu,m}$, i.e
\begin{align*}
 \chi_{\nu,m}^{(p)}=2^{\frac{\nu n}{p}} \text{ if } x \in Q_{\nu,m} \text{ and }  \chi_{\nu,m}^{(p)}=0 \text{ if } x \notin Q_{\nu,m}.
\end{align*}
\end{Definition}

\subsection{Local means}
\label{LocalMeans}
Let $N \in \N_0$ be given. We choose $k_0,k \in {\cal S}(\re)$ with compact support - e.g. $supp \ k_0, supp \ k \subset e \cdot Q_{0,0}$ for a suitable $e>0$ - such that
\begin{align}
\label{MeanMomenter}
D^{\alpha}\hat{k}(0)=0 \text{ if } |\alpha|< N,
\end{align}
while $\hat{k_0}(0)\neq 0$. Furthermore, let there be an $\varepsilon>0$ such that $\hat{k}(x) \neq 0$ for $0<|x|<\varepsilon$.
 
Such a choice is possible, see \cite[11.2]{Tri97}. We set $k_j(x):=2^{jn} k(2^jx)$ for $j\in \N$.
\begin{Satz}
\label{RychkovLokalSatz}
Let $N \in \N_0$ and $N>s$.

(i) Let $0<p\leq\infty$ and $0<q\leq \infty$. Then
\begin{align*}
  \|f|\bspq\|_{k_0,k}:= \|k_0*f|L_p(\re)\|+\left(\sum_{j=1}^{\infty} 2^{jsq}\|k_j*f|L_p(\re)\|^q\right)^{\frac{1}{q}}
\end{align*}
(modified for $q=\infty$) is an equivalent norm for $\|\cdot|\bspq\|$. It holds
\begin{align*}
\bspq=\left\{f \in {\cal S}'(\mathbb{R}^n): 
 \|f|\bspq\|_{k_0,k}<\infty \right\}.
\end{align*}

(ii) Let $0<p<\infty$ and $0<q\leq \infty$. Then
\begin{align*}
  \|f|\fspq\|_{k_0,k}:=\|k_0*f|L_p(\re)\|+\left\|\left(\sum_{j=1}^{\infty} 2^{jsq} \left|(k_j*f)(\cdot)\right|^q \right)^{\frac{1}{q}}\big|L_p(\re)\right\|
\end{align*}
(modified for $q=\infty$) is an equivalent norm for $\|\cdot|\fspq\|$. It holds
\begin{align*}
 \fspq=\left\{f \in {\cal S}'(\mathbb{R}^n): 
 \|f|\fspq\|_{k_0,k}<\infty \right\}.
\end{align*}
\end{Satz}
\begin{Bemerkung}
 This proposition is due to \cite{Ryc99}. Some minor technicalities of the proof where modified in the fourth step of \cite[Theorem 2.1]{Sch10} (for the more general vector-valued case).
\end{Bemerkung}

\subsection{A general atomic representation theorem}
We start with a lemma which helps us to understand the relation between conditions like \eqref{Atom3} and \eqref{Atom31} and which will be heavily used in the proof of the atomic representation theorem. It also shows that local means and atoms are related, see condition \eqref{MeanMomenter}.
\begin{Lemma}
 \label{Momenter}
Let $j \in \N_0$. If $k_0$ and $k_j=2^{jn}k(2^{j}\cdot)$ for $j \in \N$ are local means as in Definition \ref{LocalMeans}, then $2^{-j\left(s+n\left(1-\frac{1}{p}\right)\right)}\cdot k_j$ is an $(s,p)_{K,L}$-atom located at $Q_{j,0}$ for arbitrary $K>0$ and for $L\leq N+1$ with $N$ from \eqref{MeanMomenter}.
\end{Lemma}
\begin{proof}
 For $j=0$ there is nothing to prove since the moment condition \eqref{Atom3} follows from \eqref{Atom1} and \eqref{Atom2}. So we can concentrate on $j \in \N$: The support condition \eqref{Atom1} follows from the compact support of $k$. Furthermore,
\begin{align*}
 \|k_j(2^{-j})|\hold[K]\|=2^{jn} \|k|\hold[K]\|\leq C \, 2^{jn}
\end{align*}
since $K$ is arbitrarily often differentiable. Hence, the Hölder-condition \eqref{Atom2} is shown.

Now we have to show condition \eqref{Atom3}. There is nothing to prove for $j=0$. Hence, we can use the moment conditions \eqref{MeanMomenter}. Let $L>0$, $L=\Lint+\Lrest$ as in Section \ref{Hoelder} and let $\psi \in \hold[L]$. We expand the $\Lint$-times continuously differentiable function $\psi$ into its Taylor series of order $\Lint-1$. Then there exists a $\theta \in (0,1)$ with
\begin{align*}
 \psi(x)&=\sum_{|\beta|\leq \Lint-1} \frac{1}{\beta!} \, D^{\beta}\psi(0) \cdot x^{\beta}+\sum_{|\beta|= \Lint} \frac{1}{\beta!} \, D^{\beta}\psi(\theta x) \cdot x^{\beta}. 
\end{align*}
Hence
\begin{align*}
 \Big|\psi(x)-\sum_{|\beta|\leq \Lint} \frac{1}{\beta!} \, D^{\beta}\psi(0) \cdot x^{\beta}\Big| &=\Big|\sum_{|\beta|= \Lint} \frac{1}{\beta!} \, \big(D^{\beta}\psi(\theta x)-D^{\beta}\psi(0)\big)x^{\beta}\Big| \\
&\leq c \cdot \|\psi|\hold[L]\| \cdot |x|^{L}.
\end{align*}
Using \eqref{MeanMomenter} for $k_j$ and $\Lint\leq N$ we can insert the polynomial terms into the integral and get
\begin{align}
\label{IntMomenter}
  \Big|\int_{d \cdot Q_{j,0}} \!\!\psi(x) k_j(x) \ dx \Big| \leq c \ \|\psi|\hold[L]\| \int_{d \cdot Q_{j,0}} \!\!\!\!|k_j(x)| \cdot |x|^{L} \ dx \leq C \cdot 2^{-jL} \|\psi|\hold[L]\|.
\end{align}
Hence $2^{-j\left(s+n\left(1-\frac{1}{p}\right)\right)}k_j$ fulfils condition \eqref{Atom3}. The constant $C$ does not depend on $j \in \N_0$.
\end{proof}
\begin{Bemerkung}
\label{MomentBem}
 If we take a look at the proof, we see that instead of \eqref{MeanMomenter} it suffices to have
\begin{align}
\label{MeanMomenter2}
 \Big|\int_{d \cdot Q_{j,0}} x^{\beta} k_j(x) \ dx \Big| \leq C \cdot 2^{-jL} \text{ if } |\beta|\leq \Lint. 
\end{align}
In fact, this condition is equivalent to condition \eqref{Atom3} for $k_j$ since $\|x^{\beta} \cdot \psi| \hold[L]\|\leq C$ if $|\beta|\leq \Lint$, where $\psi \in C^{\infty}(\re)$ is a cutoff function, i.e. with compact support and $\psi(x)=1$ for $x \in supp \ k$, hence for $x \in supp \ k_j$, too.
\end{Bemerkung}

Now we will see what happens if an atom is dilated.
\begin{Lemma}
 \label{Dilation}
Let $j\in \N_0$ and $j\leq \nu$. If $a_{\nu,m}$ is an $(s,p)_{K,L}$-atom located at the cube $Q_{\nu,m}$, then $2^{j(s-\frac{n}{p})}\cdot a_{\nu,m}(2^{-j}\cdot)$ is an $(s,p)_{K,L}$-atom located at $Q_{\nu-j,m}$. 
\end{Lemma}
\begin{proof}
 The support condition \eqref{Atom1} and the Hölder-condition \eqref{Atom2}  are easy to verify. Considering the moment condition \eqref{Atom3} we have 
\begin{align*}
 \Big|\int_{d \cdot Q_{\nu-j,m}} \psi(x) a_{\nu,m}(2^{-j}x) \ dx \Big|&=2^{jn} \cdot \Big|\int_{d \cdot Q_{\nu,m}} \psi(2^jx) a_{\nu,m}(x) \ dx \Big| \\
& \leq C \cdot 2^{j n} \cdot 2^{-\nu\Kap} \cdot \|\psi(2^j\cdot)| \hold[L]\| \\
& \leq C \cdot 2^{jn} \cdot 2^{-\nu\Kap}  \cdot 2^{jL} \cdot \|\psi|\hold[L]\| \\
& = C \cdot 2^{-(\nu-j)\Kap} \cdot 2^{-j(s-\frac{n}{p})} \cdot \|\psi|\hold[L]\|.
\end{align*}
This is what we wanted to prove.
\end{proof}

Now we come to the essential part - showing the atomic representation theorem. We will use an approach as in Theorem 13.8 of \cite{Tri97}. Using the more general form of the atoms we are able to simplify the proof: One has to estimate
\begin{align*}
 \int k_j (x-y) a_{\nu,m}(y) \ dy,
\end{align*}
where $k_j$ are the local means from Section \ref{LocalMeans} and $a_{\nu,m}$ are atoms located at $Q_{\nu,m}$. One has to distinguish between $j\geq \nu$ and $j<\nu$ as in the original proof - but now  both cases can be proven very similarly with our more general approach of atoms. 

At first we prove the convergence of the atomic series in ${\cal S}'(\re)$. 
\begin{Lemma}
\label{HarmS'-KonvAtom}
Let $0<p\leq\infty$ resp. $0<p<\infty$, $0<q\leq \infty$ and $s \in \R$. Let $K\geq 0$, $L\geq 0$ with $L> \sigma_p-s$. Then 
\begin{align*}
 \sum_{\nu=0}^{\infty} \sum_{m \in \Z^n} \lambda_{\nu,m} a_{\nu,m}
\end{align*}
converges unconditionally in ${\cal S}'(\re)$, where $a_{\nu,m}$ are $(s,p)_{K,L}$-atoms located at $Q_{\nu,m}$ and $\lambda \in b_{p,q}$ or $\lambda \in f_{p,q}$.
\end{Lemma}
\begin{proof}
Let $\varphi \in {\cal S}(\re)$. Having in mind \eqref{Atom1} and \eqref{Atom3} we obtain
\begin{align*}
 \sum_{m}\left| \int_{\re}  \lambda_{\nu,m} a_{\nu,m}(x)  \varphi(x) \ dx\right| \leq C \cdot 2^{-\nu \Kap} \sum_{m} |\lambda_{\nu,m}| \cdot  \|\varphi\cdot \psi(2^{\nu} \cdot -m)|\hold[L]\|,\\ 
\end{align*}
where $\psi \in C^{\infty}(\re)$, $\psi(x)=1$ for $x \in d \cdot Q_{0,0}$ and $supp \ \psi \in (d+1) \cdot Q_{0,0}$.

Observing $\Kap=s+L+n\left(1-\frac{1}{p}\right)$ and $L>\sigma_p-s$ we get
\begin{align}
\label{Kap}
 \Kap > \begin{cases}
                0, & 0<p\leq 1 \\
		n\left(1-\frac{1}{p}\right), & 1<p\leq \infty.
               \end{cases}
\end{align}
Furthermore, since $\varphi \in {\cal S}(\re)$ we have
\begin{align*}
  \|\varphi\cdot \psi(2^{\nu} \cdot -m)|\hold[L]\| \leq C_M \cdot \left(1+|2^{-\nu}m|\right)^{-M},
\end{align*}
where $M \in \N_0$ is at our disposal and $C_M$ does not depend on $\nu$ and $m$. 

Let at first be $0<p\leq 1$. Then we choose $M=0$ and get
\begin{align*}
 \sum_{m}\left| \int_{\re}  \lambda_{\nu,m} a_{\nu,m}(x)  \varphi(x) \ dx\right| \leq C' \cdot 2^{-\nu\Kap} \sum_{m} |\lambda_{\nu,m}| \leq C' \cdot 2^{-\nu \Kap} \left(\sum_{m} |\lambda_{\nu,m}|^p\right)^{\frac{1}{p}}.
\end{align*}
Summing up over $\nu \in \N_0$ using $\Kap>0$ we finally arrive at
\begin{align}
\label{Summing}
 \sum_{\nu} \sum_{m}\left| \int_{\re}  \lambda_{\nu,m} a_{\nu,m}(x)  \varphi(x) \ dx\right| \leq C' \cdot \|\lambda|b_{p,\infty}\|. 
\end{align}

In the case $1<p\leq \infty$ we choose $M \in \N_0$ such that $Mp'>n$, where $1=\frac{1}{p}+\frac{1}{p'}$. Using Hölder's inequality we get
\begin{align*}
 \sum_{m}\left| \int_{\re}  \lambda_{\nu,m} a_{\nu,m}(x)  \varphi(x) \ dx\right| &\leq C_M \cdot 2^{-\nu\Kap} \sum_{m} |\lambda_{\nu,m}| \cdot \left(1+|2^{-\nu}m|\right)^{-M} \\
&\leq C' \cdot 2^{-\nu \Kap} \left(\sum_{m} \left(1+|2^{-\nu}m|\right)^{-Mp'}\right)^{\frac{1}{p'}} \cdot \left(\sum_{m} |\lambda_{\nu,m}|^p\right)^{\frac{1}{p}} \\
&\leq C'' \cdot 2^{-\nu \Kap} \cdot 2^{\nu \frac{n}{p'}} \cdot \left(\sum_{m} |\lambda_{\nu,m}|^p\right)^{\frac{1}{p}} .
\end{align*}
By \eqref{Kap} the exponent is smaller than zero. Hence summing over $\nu \in \N_0$ gives the same result as in \eqref{Summing}. 

Since 
\begin{align*}
b_{p,q} \hookrightarrow b_{p,\infty} \text{ resp. } 
f_{p,q} \hookrightarrow f_{p,\infty} 
\end{align*}
we have shown the absolut and hence unconditional convergence in ${\cal S}'(\re)$.

\end{proof}

\begin{Theorem}
\label{AtomicRepr}
 (i) Let $0<p\leq \infty$, $0<q\leq \infty$ and $s \in \mathbb{R}$. Let $K,L \in \R$, $K,L\geq 0$, $K>s$ and $L >\sigma_p-s$. Then $f \in {\cal S}'(\mathbb{R}^n)$ belongs to $\bspq$ if and only if it can be represented as
\begin{align*}
 f=\sum_{\nu=0}^{\infty} \sum_{m \in \Z^n}\lambda_{\nu,m}a_{\nu,m} \quad \text{with convergence in } {\cal S}'(\mathbb{R}^n).
\end{align*} 
Here $a_{\nu,m}$ are $(s,p)_{K,L}$-atoms located at $Q_{\nu,m}$ (with the same constants $d>1$ and $C>0$ in Definition \ref{Atoms} for all $\nu \in \N_0, m \in \Z$) and $\|\lambda|b_{p,q}\| < \infty$ . Furthermore, we have in the sense of equivalence of norms
\begin{align*}
 \|f|\bspq\| \sim \inf \ \|\lambda|b_{p,q}\|,
\end{align*}
where the infimum on the right-hand side is taken over all admissible representations of $f$.

(ii) Let $0<p< \infty$, $0<q\leq \infty$ and $s \in \mathbb{R}$. Let $K,L \in \R$, $K,L\geq 0$, $K>s$ and $L >\sigma_{p,q}-s$. Then $f \in {\cal S}'(\mathbb{R}^n)$ belongs to $\fspq$ if and only if it can be represented as
\begin{align*}
 f=\sum_{\nu=0}^{\infty} \sum_{m \in \Z^n}\lambda_{\nu,m}a_{\nu,m} \quad \text{with convergence in } {\cal S}'(\mathbb{R}^n).
\end{align*}
Here $a_{\nu,m}$ are $(s,p)_{K,L}$-atoms located at $Q_{\nu,m}$ (with the same constants $d>1$ and $C>0$ in Definition \ref{Atoms} for all $\nu \in \N_0, m \in \Z$) and $\|\lambda|f_{p,q}\| < \infty$. Furthermore, we have in the sense of equivalence of norms
\begin{align*}
 \|f|\fspq\| \sim \inf \ \|\lambda|f_{p,q}\|,
\end{align*}
where the infimum on the right-hand side is taken over all admissible representations of $f$.
\end{Theorem}
\begin{proof}
 We rely on the proof of Theorem 13.8 of \cite{Tri97}, now modified keeping in mind the more general conditions \eqref{Atom2} and \eqref{Atom3} instead of \eqref{Atom21} and \eqref{Atom31}. There are two directions we have to prove. 

At first, let us assume that $f$ from $\bspq$ or $\fspq$ is given. Then we know from Theorem 13.8 of \cite{Tri97} that $f$ can be written as an atomic decomposition, with atoms now fulfilling conditions \eqref{Atom21} and \eqref{Atom31} for given natural numbers $K'>s$ and $L'+1>\sigma_{p}-s$ resp. $L'+1>\sigma_{p,q}-s$. Hence, because of $C^{K'}(\re) \subset \hold[K']$, condition \eqref{Atom2} is fulfilled for all $K\leq K'$.

Conditions \eqref{Atom3} are generalizations of the classical moment conditions \eqref{Atom31} and are ordered in $L$, see Remark \ref{Ordered}. 

Thus, every classical $(s,p)_{K',L'}$-atom is an $(s,p)_{K,L}$ atom in the sense of definition \ref{Atoms} for $K\leq K'$ and $L\leq L'+1$ and this immediately shows that we find a decomposition of $f$ from $\bspq$ or $\fspq$ for arbitrary $K$ and $L$ in terms of the general atoms we introduced.

Now we come to the essential part of the proof. We have to show that, although we weakened the conditions on the atoms, a linear combination of atoms is still an element of $\bspq$ resp. $\fspq$. We modify the proof of Theorem 13.8 of \cite{Tri97} or into \cite{Sch10} where some minor technical details are modified (for the more general vector-valued case). There one uses the equivalent characterization by local means $k_0, k_j:=2^{jn} k(2^{j}\cdot)$ with a suitably large $N$(see Proposition \ref{RychkovLokalSatz}) and distinguishes between the cases $j\geq \nu$ and $j<\nu$. In both cases the crucial part is the estimate of
\begin{align*}
 \int k_j (x-y) a_{\nu,m}(y) \ dy,
\end{align*}
where $a_{\nu,m}$ is an $(s,p)_{K,L}$-atom centered at $Q_{\nu,m}$.
The idea now is to use that not only $a_{\nu,m}$ but also $k_j$ can been interpreted as atoms and admit estimates as in \eqref{Atom2} and \eqref{Atom3}, see Lemma \ref{Momenter}.

Let at first be $j\geq \nu$. The function $k$ has compact support and fulfils moment conditions \eqref{Atom31}. At first we transform the integral, having in mind the form of condition \eqref{Atom2} of $a_{\nu,m}$,
\begin{align*}
 2^{js} \int k_j (y) a_{\nu,m}(x-y) \ dy= 2^{js} \int k_{j-\nu}(y)a_{\nu,m}(x-2^{-\nu}y) \ dy. 
\end{align*}
Surely, this integral vanishes for $x \notin c \cdot Q_{\nu,m}$ for a suitable $c>0$ because of $j\geq \nu$. So we concentrate on $x\in c \cdot Q_{\nu,m}$: By Lemmata \ref{Momenter} and \ref{Dilation} the function 
\begin{align*}
2^{-(j-\nu)\left(s+n\left(1-\frac{1}{p}\right)\right)} \cdot k_{j-\nu}=2^{-(j-\nu)\left(s+n\left(1-\frac{1}{p}\right)\right)} \cdot 2^{-\nu n} \cdot k_{j}(2^{-\nu}\cdot)
\end{align*}
is an $(s,p)_{M,N}$-atom located at $Q_{j-\nu,0}$ for $M$ arbitrarily large and $N$ from \eqref{MeanMomenter}, so that also $N$ may be arbitrarily large, but fixed. Now we will use the moment condition \eqref{Atom3} for $k_{j-\nu}$ and the Hölder-condition \eqref{Atom2} for $a_{\nu,m}$. Hence, with $\psi(y)=a_{\nu,m}(x-2^{-\nu}y)$ and $N \geq K$ we have
\begin{align*}
 2^{js} \Big| \int k_{j-\nu}(y) a_{\nu,m}(x-2^{-\nu}y) \ dy \Big| &\leq C \cdot 2^{js} \cdot 2^{-(j-\nu)K} \cdot \|a_{\nu,m}(x-2^{-\nu}\cdot)|\hold[K]\| \\
&= C \cdot 2^{js} \cdot 2^{-(j-\nu)K} \cdot \|a_{\nu,m}(2^{-\nu}\cdot)|\hold[K]\| \\
&\leq C \cdot 2^{js} \cdot 2^{-(j-\nu)K} \cdot 2^{-\nu(s-\frac{n}{p})} \\
&= C' \cdot 2^{-(j-\nu)(K-s)} \cdot \chi^{(p)}(c\cdot Q_{\nu,m}),
\end{align*}
where $\chi^{(p)}(c\cdot Q_{\nu,m})$ is the $L_p(\re)$-normalized characteristic function of $c \cdot Q_{\nu,m}$. This inequality is certainly true for $x \notin c \cdot Q_{\nu,m}$. Hence (13.37) in \cite{Tri97} is shown.

Now let $j<\nu$. We will interchange the roles of $k_j$ and $a_{\nu,m}$ using condition \eqref{Atom2} now for $k_j$ and \eqref{Atom3} for $a_{\nu,m}$. Hence we start with
\begin{align*}
 2^{js} \int k_j (x-y) a_{\nu,m}(y) \ dy =2^{js} \int k(2^{j}x-y)a_{\nu,m}(2^{-j}y) \ dy.
\end{align*}
Surely, this integral vanishes for $x \notin c\cdot 2^{\nu-j}\cdot Q_{\nu,m}$. So we concentrate on $x\in c \cdot 2^{\nu-j}\cdot Q_{\nu,m}$: By Lemma \ref{Dilation} we know that $2^{j(s-\frac{n}{p})} \cdot a_{\nu,m}(2^{-j}\cdot)$ is an $(s,p)_{K,L}$-atom located at $Q_{\nu-j,m}$ while $k$ is an $(s,p)_{M,N}$-atom located at $Q_{0,0}$. Thus, using \eqref{Atom2} for $k$ with $M\geq L$, we get
\begin{align*}
 2^{js} \Big|\int k(2^{j}x-y)a_{\nu,m}(2^{-j}y) \ dy \Big| &\leq c \cdot 2^{js} \cdot 2^{-(\nu-j)\Kap} \cdot 2^{-j(s-\frac{n}{p})} \cdot \|k(2^{j}x-\cdot)|\hold[L]\| \\
&\leq c \cdot 2^{-(\nu-j)(L+s)} \cdot 2^{\nu\frac{n}{p}} \cdot 2^{-(\nu-j)n}  \\
&= c \cdot 2^{-(\nu-j)(L+s)} \cdot 2^{\nu\frac{n}{p}} \cdot 2^{-(\nu-j)n}\cdot \chi(c \cdot 2^{\nu-j}\cdot Q_{\nu,m}).
\end{align*}
where $\chi(c \cdot 2^{\nu-j}\cdot Q_{\nu,m})$ is the characteristic function of $c \cdot 2^{\nu-j}\cdot Q_{\nu,m}$. This estimate is the same as (13.41) combined with (13.42) in \cite{Tri97} or (72) and (73) in \cite{Sch10}, observing that we use $L$ instead of $L+1$ in the atomic representation theorem. 

Starting with these two estimates we can follow the steps in \cite{Tri97} or \cite{Sch10} and finish the proof, since $K>s$ and $L>\sigma_p-s$ resp. $L>\sigma_{p,q}-s$. Strictly speaking, we arrive (in the $\bspq$-case) at
\begin{align*}
 \Big\|\sum_{\nu\leq \nu_0} \sum_{|m|\leq m_0}\lambda_{\nu,m} a_{\nu,m} \big|\bspq\Big\|\leq C \cdot \|\lambda|b_{p,q}\|
\end{align*}
for all $\nu_0,m_0 \in \N_0$ with a constant $C$ independent of $\nu_0$ and $m_0$. Using Lemma \ref{HarmS'-KonvAtom} and the Fatou property of the spaces $\bspq$ resp. $\fspq$ (see Proposition \ref{Fatou}) we are finally done, i.e.
\begin{align*}
 \Big\|\sum_{\nu\in \N_0} \sum_{m \in \Z^n}\lambda_{\nu,m} a_{\nu,m} \big|\bspq\Big\|\leq C \cdot \|\lambda|b_{p,q}\|.
\end{align*}
\end{proof}

\begin{Bemerkung}
 The conditions $\eqref{Atom2}$ and $\eqref{Atom3}$ for the atomic representation theorem can be slightly modified: If $K>0$, then it is possible to replace $\|\cdot|\hold[K]\|$ by $\|\cdot|B_{\infty,\infty}^{K}(\re)\|$ in condition $\eqref{Atom2}$. This is clear for $K \notin \N$, see Remark \ref{HoldB}. If $K \in \N$, this follows from 
\begin{align*}
\hold[K]  \hookrightarrow B_{\infty,\infty}^{K}(\re) \hookrightarrow \hold[K-\varepsilon]
\end{align*}
for $\varepsilon>0$. 

A similar result holds true for $L>0$, $L \notin \N$ and condition \eqref{Atom3} by trivial means. If $L \in \N$, then $\|\cdot|\hold[L]\|$ can be replaced by $\|\cdot|C^L(\re)\|$, where the condition needs to be true for all $\psi \in C^L(\re)$. This follows from the fact, that both conditions imply \eqref{Atom32}. Hence they are equivalent. 

It is not clear to the author whether $\|\cdot|\hold[L]\|$ can be replaced by $\|\cdot|B_{\infty,\infty}^{L}(\re)\|$ for $L \in \N$. 

\end{Bemerkung}

\begin{Bemerkung}
 In the proof of Theorem \ref{AtomicRepr} we assumed that the local means $k_j$ are arbitrarily often differentiable and fulfil as many moment conditions as we wanted. But if we take a look into the proof, we see that we did not use the specific structure $k_j=2^{jn}k(2^j\cdot)$. It is sufficient to know that there are constants $c$ and $C$ such that for all $j\in \N_0$ it holds $supp \ k_j \subset c\cdot Q_{j,0}$, that
\begin{align}
 \label{LocalMean2}
\|k_j(2^{-j}\cdot)|\hold[M]\| &\leq C \cdot  2^{jn}
\end{align}
with $M \geq L$ and that for every $\psi \in \hold[N]$ it holds 	
\begin{align}
\label{LocalMean3} \left| \int_{d \cdot Q_{j,0}} \psi(x) k_j(x) \ dx \right|\leq C \cdot  2^{-jN} \cdot \|\psi|\hold[N]\|
\end{align}
with $N \geq K$ because the atomic conditions \eqref{Atom2} and \eqref{Atom3} are ordered in $N$ and $M$, see Remark \ref{Ordered}. As before, condition \eqref{LocalMean3} can be strengthened by
\begin{align*}
 \int x^{\beta}k_j(x) \ dx = 0 \text{ for all } |\beta|< N.
\end{align*}
Through these considerations the idea arises how to prove a counterpart of Theorem \ref{AtomicRepr} for the local mean characterization in \cite[Theorem 1.15]{Tri08} without further substantial efforts. This is done in the following Corollary, including some technical issues concerning the definition of a dual pairing (see \cite[Remark 1.14]{Tri08}). It is obvious that the original version of Theorem 1.15 in \cite{Tri08} is just some kind of modification of this Corollary. 
\end{Bemerkung}
\begin{Folgerung}
\label{FolgLocal}
(i) Let $0<p\leq \infty$, $0<q\leq \infty$ and $s \in \mathbb{R}$. Let $M,N\in \R$, $M,N\geq 0$, $M>\sigma_p-s$ and $N>s$. Assume that for all $j \in \N_0$ it holds that $k_j \in \hold[M]$, $supp \ k_j \subset c\cdot Q_{j,0}$ and $k_j$ fulfils  \eqref{LocalMean2} and \eqref{LocalMean3}. Then there is a constant c such that
\begin{align*}
  \|f|\bspq\|_k:=\|k_0*f|L_p(\re)\|+\left(\sum_{j=1}^{\infty} 2^{jsq}\|k_j*f|L_p(\re)\|^q\right)^{\frac{1}{q}} \leq c \cdot \|f|\bspq\|
\end{align*}
(modified for $q=\infty$) for all $f \in \bspq$.

(ii) Let $0<p< \infty$, $0<q\leq \infty$ and $s \in \mathbb{R}$. Let $M,N\in \R$, $M,N\geq 0$, $M>\sigma_{p,q}-s$ and $N>s$. Assume that for all $j \in \N_0$ it holds that  $k_j \in \hold[M]$, $supp \ k_j \subset c\cdot Q_{j,0}$ and $k_j$ fulfils \eqref{LocalMean2} and \eqref{LocalMean3}. Then there is a constant c such that
\begin{align*}
  \|f|\fspq\|_k:&=\|k_0*f|L_p(\re)\|+\bigg\|\Big(\sum_{j=1}^{\infty} 2^{jsq} \left|(k_j*f)(\cdot)\right|^q \Big)^{\frac{1}{q}}\Big|L_p(\re)\bigg\| \\
 &\leq c \cdot \|f|\fspq\|
\end{align*}
(modified for $q=\infty$) for all $f \in \fspq$.
\end{Folgerung}
\begin{proof}
 There is nearly nothing left to prove because the crucial steps were done in the proof before: Let $f \in \bspq$ (analogously for $f \in \fspq$) be given. By Theorem \ref{AtomicRepr} we can represent $f \in \bspq$ by an ''optimal'' atomic decomposition 
\begin{align*}
 f=\sum_{\nu=0}^{\infty} \sum_{m \in \mathbb{Z}^n}\lambda_{\nu,m}a_{\nu,m}, 
\end{align*}
where $a_{\nu,m}$ is an $(s,p)_{N,M}$-atom located at $Q_{\nu,m}$ and $\|f|\bspq\| \sim \|\lambda|b_{p,q}\|$ (with constants independent of $f$).

But, by the second step of the proof of Theorem \ref{AtomicRepr} and the considerations in the succeeding remark we have

 \begin{align}
\label{AbschAtom}
 \Big\|\sum_{\nu\leq \nu_0} \sum_{|m|\leq m_0}\lambda_{\nu,m} a_{\nu,m} \big|\bspq \Big\|_k\leq C \cdot \|\lambda|b_{p,q}\| \sim \|f|\bspq\|
\end{align}
for all $\nu_0,m_0 \in \N_0$ with a constant $C$ independent of $\nu_0$ and $m_0$. 

Finally, we use a similar duality argument as in \cite[Remark 1.14]{Tri08} or \cite[Section 5.1.7]{Tri06} to justify the dual pairing of $k_j$ and $f$. Looking into the proof of Lemma \ref{HarmS'-KonvAtom}, we see that
\begin{align}
\label{KonvB}
 \sum_{\nu} \sum_{m}\left| \int_{\re}  \lambda_{\nu,m} a_{\nu,m}(x) \varphi(x) \ dx\right| \leq C' \cdot \|\varphi|\hold[M-\varepsilon]\|\cdot \|\lambda|b_{p,\infty}\|
\end{align}
for $\varphi \in \hold[M]$ with compact support, $M-\varepsilon\geq 0$ and $M-\varepsilon>\sigma_p-s$, where $C'$ depends on the support of $\varphi$. This includes the functions $k_j$ for $j \in \N_0$. Because of this absolut convergence the dual pairing of $f$ and $\varphi$ is given by
\begin{align*}
 \lim_{m_0,\nu_0 \rightarrow \infty} \, \sum_{\nu\leq \nu_0} \sum_{|m|\leq m_0 } \int_{\re}  \lambda_{\nu,m} a_{\nu,m}(x) \varphi(x) \ dx .
\end{align*}

Furthermore, for two different atomic decompositions of $f$ these limits are the same: By definition of a distribution $f \in {\cal S}'(\re)$ and Lemma \ref{HarmS'-KonvAtom} this is valid for $\varphi \in {\cal S}(\re)$. For arbitrary $\varphi \in \hold[M]$ with compact support this follows by \eqref{KonvB} and density arguments because $C^{\infty}(\re)$ is dense in $\hold[M]$ with respect to the norm of $\hold[M-\varepsilon]$. For instance, this can be seen using
\begin{align*}
 \hold[M] \hookrightarrow B_{\infty,\infty}^M(\re) \hookrightarrow B_{\infty,q}^{M-\varepsilon}(\re) \hookrightarrow  B_{\infty,\infty}^{M-\varepsilon}=\hold[M-\varepsilon]
\end{align*}
for $M-\varepsilon \notin \N_0$ and the fact that $C^{\infty}(\re)$ is dense in $\bspq$ if $q<\infty$. 

Hence we have 
\begin{align*}
 \sum_{\nu\leq \nu_0} \sum_{|m|\leq m_0}\lambda_{\nu,m} \left(a_{\nu,m} * k_j\right)(x) \rightarrow (f * k_j) (x) \text{ for } \nu_0, m_0 \rightarrow \infty 
\end{align*}
for all $x \in \re$. Using the standard Fatou lemma and \eqref{AbschAtom} we finally get 
\begin{align*}
 \|f|\bspq\|_k\leq C \cdot \|\lambda|b_{p,q}\| \sim \|f|\bspq\|.
\end{align*}
\end{proof}

\section{Key theorems}
\subsection{Pointwise multipliers}
\label{PointwiseSect}
Triebel proved in Section 4.2 of \cite{Tri92} the following assertion.
\begin{Theorem}
\label{pointwise}
 Let $s \in \mathbb{R}$ and $0<q\leq \infty$.

(i) Let $0<p\leq \infty$ and $\rho > \max(s,\sigma_p-s)$. Then there exists a positive number $c$ such that
\begin{align*}
 \|\varphi f|\bspq\| \leq c \|\varphi|\hold[\rho]\| \cdot \|f|\bspq\|
\end{align*}
for all $\varphi \in \hold[\rho]$ and all $f \in \bspq$.

(ii) Let $0<p<\infty$ and $\rho > \max(s,\sigma_{p,q}-s)$. Then there exists a positive number $c$ such that
\begin{align*}
 \|\varphi f|\fspq\| \leq c \|\varphi|\hold[\rho]\| \cdot \|f|\fspq\|
\end{align*}
for all $\varphi \in \hold[\rho]$ and all $f \in \fspq$.
\end{Theorem}
He excluded the cases $\rho \in \N$. This is not necessary in our considerations.

The very first idea to prove this result is to take an atomic decomposition of $f$, to multiply it by $\varphi$ and to prove that the resulting sum is again a sum of atoms. Hence one has to check whether a product of an $(s,p)_{K,L}$-atom and a function $\varphi$ is still an $(s,p)_{K,L}$-atom. 

But there was a problem: Moment conditions like $\eqref{Atom31}$ are (in general) destroyed by multiplication with $\varphi$. So the atomic approach in \cite{Tri92} only worked when no moment conditions were required, hence if $s>\sigma_p$ resp. $s>\sigma_{p,q}$, and the full generality of Theorem \ref{pointwise} had to be obtained by an approach via local means. Looking at condition $\eqref{Atom3}$ instead the situation when multiplying by $\varphi$ is now different. 

Furthermore, the atomic approach only worked for $\varphi \in C^{k}(\re)$ with $k \in \N$ and $k>s$ having in mind condition \eqref{Atom21}. Now we are able to give weaker conditions using the new atomic approach with condition \eqref{Atom2}.

We start with a first standard analytical observation.
\begin{Lemma}
\label{helpHoelder}
 Let $s\geq 0$. There exists a constant $c>0$ such that for all $f,g \in \hold$ the product $f\cdot g$ belongs to $\hold$ and it holds
\begin{align*}
  \|f\cdot g|\hold\| \leq  c \cdot \|f|\hold\| \cdot  \|g|\hold\|. 
\end{align*}
\end{Lemma}
\begin{bew}
This can be proven using standard arguments, especially Leibniz formula.
\end{bew}
Now we are ready to prove Theorem \ref{pointwise}. This is done by the following lemma together with Theorem \ref{AtomicRepr} using the mentioned technique of atomic decompositions. For some further technicalities see the upcoming Remark \ref{ProdDef} or \cite[4.2.2, Remark 1] {Tri92}. This covers also the well-definedness of the product.

\begin{Lemma}
\label{ProdAtom}
 There exists a constant $c$ with the following property: For all $\nu \in \N_0$, $m \in \Z$, all $(s,p)_{K,L}$-atoms $a_{\nu,m}$ with support in $d \cdot Q_{\nu,m}$ and all $\varphi \in \hold[\rho]$ with $\rho \geq \max(K,L)$ the product
\begin{align*}
 c \cdot \|\varphi|\hold[\rho]\|^{-1} \cdot \varphi \cdot a_{\nu,m} 
\end{align*}
is an $(s,p)_{K,L}$-atom with support in $d\cdot Q_{\nu,m}$.
\end{Lemma}
\begin{bew}
Regarding the conditions \eqref{Atom2} on the derivatives Lemma \ref{helpHoelder} gives
\begin{align*}
 \|(\varphi \cdot a)(2^{-\nu}\cdot)|\hold[K]\| &\leq  c \cdot \|\varphi(2^{-\nu}\cdot)|\hold[K]\| \cdot \|a(2^{-\nu}\cdot)|\hold[K]\|  \\
&\leq c' \cdot \|\varphi|\hold[K]\| \cdot 2^{-\nu(s-\frac{n}{p})}. 
\end{align*}
Now we come to the preservation of the moment conditions \eqref{Atom3}. By our assumptions there exists a constant $C>0$ such that for every $\psi \in \hold[L]$ it holds 	
\begin{align*}
\left| \int_{d \cdot Q_{\nu,m}} \psi(x) a(x) \ dx \right|\leq C \cdot  2^{-\nu\Kap} \|\psi|\hold[L]\|.
\end{align*}
Using this inequality now for $\psi \cdot \varphi$ instead of $\psi$ together with Lemma \ref{helpHoelder} it follows
\begin{align*}
\left| \int_{d \cdot Q_{\nu,m}} \psi(x) \big(\varphi(x)\cdot a(x)\big) \ dx \right|&=\left| \int_{d \cdot Q_{\nu,m}} \big(\psi(x) \cdot \varphi(x)\big)a(x) \ dx \right| \\
&\leq C \cdot 2^{-\nu\Kap} \|\psi\cdot\varphi|\hold[L]\| \\
&\leq C' \cdot 2^{-\nu\Kap} \|\psi|\hold[L]\| \cdot \|\varphi|\hold[L]\|.
\end{align*}
Hence our lemma is shown.
\end{bew}
\begin{Bemerkung}
 This is the more general version of part 1 of Lemma 1 in \cite{Skr98} using now the wider atomic approach from \ref{Atoms} which yields a stronger result than in \cite{Skr98}.  
\end{Bemerkung}

\begin{Bemerkung}
 \label{ProdDef}
As at the end of Corollary \ref{FolgLocal} we have to deal with some technicalities. We concentrate on the $\bspq$-case, the $\fspq$-case is nearly the same. In principle, Lemma \ref{ProdAtom} shows that
\begin{align}
\label{AtomKonv}
 \sum_{\nu} \sum_{m}\lambda_{\nu,m} (a_{\nu,m} \cdot \varphi) 
\end{align}
converges unconditionally in ${\cal S}'(\re)$ where
\begin{align*}
 f=\sum_{\nu} \sum_{m}\lambda_{\nu,m} a_{\nu,m} \text{ in } {\cal S}'(\re)
\end{align*}
and the limit belongs to $\bspq$ if $f$ belongs to $\bspq$. 

To define the product of $\varphi$ and $f$ as this limit, we have to show that the limit does not depend on the atomic decomposition we chose for $f$. 

Hence we are pretty much in the same situation as at the end of Corollary \ref{FolgLocal}: Let at first be $\varphi \in C^{\infty}$. Then the multiplication with $\varphi$ is a continuous operator mapping  ${\cal S}'(\re)$ to ${\cal S}'(\re)$. So \eqref{AtomKonv} converges to $\varphi \cdot f$ for all choices of atomic decompositions of $f$. Using Lemma \ref{ProdAtom} and the standard Fatou lemma we get
\begin{align*}
 \|\varphi \cdot f|\bspq\| \leq c \cdot \|\varphi|\hold[\rho]\| \cdot \|f|\bspq\|
\end{align*}
for all $f \in \bspq$. 

For arbitrary $\varphi \in \hold[\rho]$ we use a density argument similar to that at the end of Corollary \ref{FolgLocal}. We know 
\begin{align*}
 \|\varphi^* \cdot f|\bspq\| \leq c \cdot \|\varphi^*|\hold[\rho-\varepsilon]\| \cdot \|f|\bspq\|
\end{align*}
for $\varphi^* \in C^{\infty}(\re)$, $\rho$ as in Lemma \ref{ProdAtom} and $\varepsilon$ small enough. Now using the density of $C^{\infty}(\re)$ in $\hold[\rho]$ with respect to the norm of $\hold[\rho-\varepsilon]$ the uniqueness of the product and 
\begin{align*}
 \|\varphi \cdot f|\bspq\| \leq c \cdot \|\varphi|\hold[\rho-\varepsilon]\| \cdot \|f|\bspq\| \leq  c \cdot \|\varphi|\hold[\rho]\| \cdot \|f|\bspq\|
\end{align*}
follows.

\end{Bemerkung}

\begin{Bemerkung}
 Since   
\begin{align*}
 \hold[L] \hookrightarrow B_{\infty,\infty}^{L}(\re) \hookrightarrow \hold[L-\varepsilon]
\end{align*}
for $L-\varepsilon \geq 0$, we can replace $\|\varphi|\hold[\rho]\|$ by $\|\varphi|B_{\infty,\infty}^{\rho}(\re)\|$, even by $\|\varphi|B_{\infty,q}^{\rho}(\re)\|$ for arbitrary $0<q \leq \infty$ in Lemma \ref{pointwise}.

The condition $\rho > \max(s,\sigma_{p,q}-s)$ for the $\fspq$-spaces in Theorem \ref{pointwise} can be replaced by $\rho > \max(s,\sigma_{p}-s)$. This is a matter of complex interpolation, see the proof of the corollary in Section 4.2.2 of \cite{Tri92}.
\end{Bemerkung}

\begin{Bemerkung} 
Our Theorem \ref{pointwise} is a special case of Theorem 4.7.1 in \cite{RS96}: By Remark \ref{HoldB} it holds $\hold[\rho]=B_{\infty,\infty}^{\rho}(\re)$ for $\rho>0$ and $\rho \notin \N$. So, let $f \in \bspq$ or $f \in \fspq$ as well as $\varphi \in \hold[\rho]$ with $\rho>s$. Then $\varphi \in B_{\infty,\infty}^{\rho'}(\re)$ for $s<\rho'<\rho$. By Theorem 4.7.1 
of \cite{RS96} it holds 
\begin{align*}
 \bspq \cdot B_{\infty,\infty}^{\rho'}(\re) \hookrightarrow \bspq \quad \text{resp.}\quad \fspq \cdot B_{\infty,\infty}^{\rho'}(\re) \hookrightarrow \fspq
\end{align*}
if
\begin{align*}
 \rho'>s \quad \text{and} \quad s+\rho'> \sigma_p \quad \Leftrightarrow \quad \rho'>s \quad \text{and} \quad \rho'>\sigma_p-s.
\end{align*}
In case of $\bspq$ these are the same conditions as in Theorem \ref{pointwise} - in case of $\fspq$ these are even better (no dependency on $q$).

It was not the idea of this paper to give such a detailed and comprehensive treatise as in Runst' and Sickel's book \cite{RS96} but to show an application of the more general atomic decompositions where the proof is easy to follow (see Triebel \cite[Section 4.1]{Tri92}). 
\end{Bemerkung}

\subsection{Diffeomorphisms}
We want to study the behaviour of the mapping
\begin{align*}
 D_{\varphi}: f \mapsto f(\varphi(\cdot)), 
\end{align*}
where $f$ is an element of the function space $\bspq$ resp. $\fspq$ and $\varphi: \re \rightarrow \re$ is a suitably smooth map. 

One would like to deal with this problem analogously to the pointwise multiplier problem in Section \ref{PointwiseSect}. Hence we start with an atomic decomposition of $f$ and composed with $\varphi$. Then we are confronted with functions of the form $a_{\nu,m}\circ \varphi$ originating from the atoms $a_{\nu,m}$. This was the idea of Section 4.3.1 in Triebel \cite{Tri92}. But in general, moment conditions of type \eqref{Atom31} are destroyed by this operator. So $s>\sigma_p$ resp. $s>\sigma_{p,q}$ was necessary. As we will see, conditions like \eqref{Atom3} behave more friendly under diffeomorphisms. 
 
Furthermore, we are confronted with more difficulties than in section \ref{PointwiseSect} because the support of an atom changes remarkably. In particular, after composing with $\varphi$ two or more atoms can be associated with the same cube $Q_{\nu,m}$ which is not possible in the atomic representation theorem \ref{AtomicRepr}. This has not been considered in detail in Section 4.3.1 by Triebel \cite{Tri92} while there is some work done in the proof of Lemma 3 by Skrzypczak \cite{Skr98}.

The special case of bi-Lipschitzian maps, also called Lipschitz diffeomorphisms, is treated in Section 4.3 by Triebel \cite{Tri02}. The main theorem there is used to obtain results for characteristic functions of Lipschitz domains as pointwise multipliers in $\bspq$ and $\fspq$.

\begin{Definition}
\label{LipDef}
Let $\rho\geq 1$. 

(i) Let $\rho=1$. We say that the map $\varphi: \re \rightarrow \re$ is a $\rho$-diffeomorphism if $\varphi$ is a bi-Lipschitzian map, i.e. that there are constants $c_1,c_2>0$ such that
\begin{align}
 \label{biLip}
 c_1\leq \frac{|\varphi(x)-\varphi(y)|}{|x-y|} \leq c_2. 
\end{align}
for all $x,y \in \re$ with $0<|x-y|\leq 1$. 

(ii) Let $\rho>1$. We say that the one-to-one map $\varphi: \re \rightarrow \re$ is a $\rho$-diffeomorphism if the components $\varphi_i$ of $\varphi(x)=(\varphi_1(x),\ldots,\varphi_n(x))$ have classical derivatives up to order $\rint$ with $\frac{\partial \varphi_i}{\partial x_j} \ \in \hold[\rho-1]$ for all $i,j\in \{1,\ldots,n\}$ and if $|\det J(\varphi)(x)| \geq c$ for some $c>0$ and all $x \in \re$. Here $J(\varphi)(x)$ stands for the Jacobian matrix of $\varphi$ at the point $x \in \re$. 
\end{Definition}
\begin{Bemerkung}
 It does not matter, whether we assume $\eqref{biLip}$ for all $x,y \in \re$ with $x\neq y$ or for all $x,y \in \re$ with $0<|x-y|<c$ for a constant $c>0$. This is obvious for the upper bound. For the lower bound we have to use the upper bound of the bi-Lipschitzian property of the inverse $\varphi^{-1}$ of $\varphi$. Its existence independent of the given exact definition of a bi-Lipschitzian map is shown in the following lemma.
\end{Bemerkung}

\begin{Lemma}
\label{DiffRem}
 Let $\rho\geq 1$.

(i) If $\varphi$ is a $1$-diffeomorphism, then $\varphi$ is bijective and $\varphi^{-1}$ is a $1$-diffeomorphism, too.

(ii) Let $\rho>1$. If $\varphi$ is a $\rho$-diffeomorphism, then its inverse $\varphi^{-1}$ is a $\rho$-diffeomorphism as well. 

(iii) If $\varphi$ is a $\rho$-diffeomorphism, then $\varphi$ is a $\rho'$-diffeomorphism for $1\leq \rho' \leq \rho$. Hence $\varphi$ is a bi-Lipschitzian map.
\end{Lemma}
\begin{bew}
To prove part (i) we use Brouwer's invariance of domain theorem (see \cite{Bro12}): Since $\varphi: \re \rightarrow \re$ is continuous and injective, the image $\varphi(U)$ of $U$ is an open set if $U$ is open. Otherwise, if $U$ is closed, then also $\varphi(U)$ is closed: If $\varphi(x_n) \rightarrow y$ with $x_n \in U$, then $x_n$ converges to some $x \in U$ by \eqref{biLip} and hence $\varphi(x_n) \rightarrow \varphi(x)=y$. Thus $\varphi$ maps $\re$ to $\re$. The inverse $\varphi^{-1}$ is automatically a bi-Lipschitzian map, see \eqref{biLip}.

The proof of observation (iii) for $\rho'>1$ is trivial. Hence, we have to show that every $\rho$-diffeomorphism is a bi-Lipschitzian map for $\rho>1$. The estimate
\begin{align*}
 \frac{|\varphi(x)-\varphi(y)|}{|x-y|} \leq c_2
\end{align*}
follows from the fact that the derivatives $\frac{\partial \varphi_i}{\partial x_j}$ are bounded for all $i,j\in \{1,\ldots,n\}$. The formula
\begin{align}
\label{Det}
 J(\varphi^{-1})(\varphi(x))=\left(J(\varphi)(x)\right)^{-1}
\end{align}
and $|\det J(\varphi)(x)| \geq c$ together show that the derivatives of the inverse $\frac{\partial (\varphi^{-1})_i}{\partial x_j}$ are bounded for all $i,j\in \{1,\ldots,n\}$, for instance using the adjugate matrix formula. By the mean value theorem there exists a $c>0$ such that
\begin{align*}
 \frac{|\varphi^{-1}(x)-\varphi^{-1}(y)|}{|x-y|} \leq c
\end{align*}
and so part (iii) is shown.

Finally, for (ii) we have to show that $\frac{\partial (\varphi^{-1})_i}{\partial x_j} \ \in \hold[\rho-1]$ and $|\det J(\varphi^{-1})(x)| \geq c$ for $\rho>1$. The latter part follows from $\eqref{Det}$ and the boundedness of $\frac{\partial \varphi_i}{\partial x_j}$. For the first we have to argue inductively in the same way as in the inverse function theorem, starting with
\begin{align*}
 J(\varphi^{-1})(x)=\left(J(\varphi)(\varphi^{-1}(x))\right)^{-1}
\end{align*}

It is well known that
\begin{align*}
 A \rightarrow A^{-1}
\end{align*}
is a $C^{\infty}(\R^{n\times n})$-mapping for invertible $A$. Together with the upcoming Lemma \ref{HoelderDiff} this shows: If the components of $J(\varphi)$ belong to $\hold[\rho-1]$ and $\varphi^{-1}$ is an $l$-diffeomorphism, then the components of $J(\varphi^{-1})$ belong to $\hold[\min(\rho-1,l)]$ and hence $\varphi^{-1}$ is a $\min(l+1,\rho)$-diffeomorphism. This inductive argument and the induction starting point that $\varphi^{-1}$ is a $1$-diffeomorphism (by part (i) and (iii)) prove that $\varphi^{-1}$ is a $\rho$-diffeomorphism. Thus the lemma is shown.  
\end{bew}

We go on with a second standard analytical observation.
\begin{Lemma}
 \label{HoelderDiff}
Let $\varphi$ be a $\rho$-diffeomorphism and let $\max(1,s)\leq \rho$. Then there exists a constant $C$ depending on $\rho$ such that for all $f \in \hold$ it holds
\begin{align*}
 \|f\circ \varphi|\hold\|\leq C_{\varphi} \cdot \|f|\hold\|.
\end{align*}
\end{Lemma}
\begin{bew}
By definition
\begin{align*}
 \|f \circ \varphi |\hold\|= \|f \circ \varphi |C^{\sint}(\re)\|+\sum_{|\alpha|=\sint} \|D^{\alpha} \left[f \circ \varphi\right]|\lip[\srest]\|.
\end{align*}
The lemma follows now by using the chain rule and Leibniz rule for spaces of differentiable functions and for Hölder spaces $\hold$.
\end{bew}
\begin{Bemerkung}
\label{DiffUni} 
As one can easily see, the constant in Lemma \ref{HoelderDiff} depends on $\sum\limits_{i=1}^n \sum\limits_{j=1}^n \left\|\frac{\partial \varphi_i}{\partial x_j}|\hold[\rho-1]\right\|$. If we have a sequence of functions $\{\varphi^m\}_{m \in \N}$ and 
\begin{align*}
\sup_{m \in \N}\sum_{i=1}^n\sum_{j=1}^n  \left\|\frac{\partial \varphi_i^m}{\partial x_j}\big|\hold[\rho-1]\right\|< \infty,
\end{align*} 
then there is a universal constant $C$ with $C_{\varphi_m}\leq C$, i.e. for all $m \in \N$ it holds
\begin{align*}
 \|f\circ \varphi_m|\hold\|\leq C \cdot \|f|\hold\|.
\end{align*}
\end{Bemerkung}

\begin{Lemma}
\label{Lpdiff}
 Let $0< p \leq \infty$. Let $\varphi: \re \rightarrow \re$ be bijective and let there be a constant $c>0$ such that
\begin{align}
\label{Lip2}
 c \leq \frac{|\varphi(x)-\varphi(y)|}{|x-y|}
\end{align}
for $x,y \in \re$ with $x\neq y$.
Then there is a constant $C>0$ such that
\begin{align}
\label{DiffLp}
 \|f \circ \varphi|L_p(\re) \| \leq C \cdot \|f\|L_p(\re)\|.
\end{align}
\end{Lemma}
\begin{bew}
If $p<\infty$, it suffices to prove \eqref{DiffLp} for 
\begin{align*}
 f=\sum_{j=1}^N a_j \chi_{A_j},
\end{align*}
where $a_j \in \C$, $A_j$ are pairwise disjoint rectangles in $\re$ and $\chi_{A_j}$ is the characteristic function of $A_j$. We have
\begin{align*}
 \int |(f \circ \varphi)(x)|^p \ dx = 
 \int \Big|\sum_{j=1}^N a_j \chi_{\varphi^{-1}(A_j)}(x)\Big|^p \ dx= \sum_{j=1}^N |a_j|^p \mu(\varphi^{-1}(A_j))
\end{align*}
because the preimages $\varphi^{-1}(A_j)$ are also pairwise disjoint.
Hence we have to show:

\textit{
There is a constant $C>0$ such that for all rectangles $A$ it holds
\begin{align}
\label{DiffMu}
 \mu(\varphi^{-1}(A)) \leq C \cdot \mu(A).
\end{align}
}
To prove this let $B_r(x_0)=\{x \in \re: |x-x_0|<r\}$ be the open ball around $x_0 \in \re$ with radius $r>0$. Then by \eqref{Lip2} we have
\begin{align}
\label{Balls}
 \varphi^{-1}(B_r(x_0)) \subset B_{\frac{r}{c}} (\varphi^{-1}(x_0)).
\end{align}
Hence there is a constant $C>0$ such that
\begin{align*}
 \mu(\varphi^{-1}(B_r(x_0)))< C \cdot \mu(B_r(x_0))
\end{align*}
for all $x_0 \in \re$, $r>0$.

Now, we cover a given rectangle $A$ with finitely many open balls $\{B_j\}_{j=1}^M$ such that
\begin{align}
\label{Covering}
 \mu\left(\bigcup_{j=1}^M B_j\right) \leq 2\mu(A). 
\end{align}
Afterwards we make use of the following Vitali covering lemma: There exists a subcollection $B_{j_1},\ldots,B_{j_m}$ of these balls which are pairwise disjoint and satisfy
\begin{align*}
  \bigcup_{j=1}^M B_j \subset \bigcup_{k=1}^m 3\cdot B_{j_k}.
\end{align*}
Using this, \eqref{Covering} and \eqref{Balls} for the balls $3 \cdot B_{j_k}$ finally gives
\begin{align*}
 \mu(\varphi^{-1}(A)) &\leq \mu\left(\varphi^{-1}\left(\bigcup_{j=1}^M B_j\right)\right)\leq \mu\left(\varphi^{-1}\left(\bigcup_{k=1}^M 3\cdot B_{j_k}\right)\right)
=\mu\left(\bigcup_{k=1}^M \varphi^{-1}(3\cdot B_{j_k})\right) \\
&=\sum_{k=1}^M \mu(\varphi^{-1}(3\cdot B_{j_k})) \leq C \cdot \sum_{k=1}^M \mu(3\cdot B_{j_k})\leq C \cdot 3^n  \cdot \sum_{k=1}^M \mu(B_{j_k})\\
& \leq 2C \cdot  3^n \cdot \mu(A).
\end{align*}
This proves the result for $0<p<\infty$.

For $p=\infty$ we have to show
\begin{align*}
 \|f\circ \varphi|L_{\infty}(\re)\| \leq \|f|L_{\infty}(\re)\|.
\end{align*}
This follows from: If $\mu(\{x \in \re: |f(x)|>a)\})=0$, then also $\mu(\{x \in \re: |f(\varphi(x))|>a\})=0$,
which is a consequence of \eqref{DiffMu}:

\textit{Let $M$ be a measurable set with $\mu(M)=0$. Then also $\mu(\varphi^{-1}(M))=0$.}

Hence the lemma is shown for $p=\infty$, too.

\end{bew}
\begin{Bemerkung}
 A proof of a more general observation using the Radon-Nikodym derivative and the Lebesgue point theorem can be found in Corollary 1.3 and Theorem 1.4 of \cite{Vod89} - but here we wanted to give a direct, more instructive proof for our special situation. 
\end{Bemerkung}

\begin{Bemerkung}
 By the previous proof it is obvious that Condition \eqref{DiffMu} is equivalent to \eqref{DiffLp} for $0<p<\infty$. Condition \eqref{DiffMu} does not depend on $p$. For Condition \eqref{DiffMu} it is necessary that the measure $m$ with $m(A):=\mu(\varphi^{-1}(A))$ is absolutely continuous with respect to the Lebesgue measure $\mu$. In case of $p=\infty$ this condition is also sufficient for \eqref{DiffMu} by the previous proof.
\end{Bemerkung}

Now we are ready for the main theorem of this section.
\begin{Theorem}
\label{Diffeo}
 Let $s \in \mathbb{R}$, $0<q\leq \infty$ and $\rho\geq 1$.

(i) Let $0<p\leq \infty$ and $\rho > \max(s,1+\sigma_p-s)$. If $\varphi$ is a $\rho$-diffeomorphism, then there exists a constant $c$ such that
\begin{align*}
 \|f(\varphi(\cdot))|\bspq\| \leq c \cdot \|f|\bspq\|
\end{align*}
for all $f \in \bspq$. Hence $D_{\varphi}$ maps $\bspq$ onto $\bspq$.

(ii) Let $0<p<\infty$ and $\rho > \max(s,1+\sigma_{p,q}-s)$. If $\varphi$ is a $\rho$-diffeomorphism, then there exists a constant $c$ such that
\begin{align*}
 \|f(\varphi(\cdot))|\fspq\| \leq c \cdot \|f|\fspq\|
\end{align*}
for all $f \in \fspq$. Hence $D_{\varphi}$ maps $\fspq$ onto $\fspq$.
\end{Theorem}
\begin{bew}
At first, beside the two conditions \eqref{Atom2} and \eqref{Atom3} we need to take a closer look at the centres and supports of the atoms. Briefly speaking, the decisive local properties of the set of atoms $a_{\nu,m}$ are maintained by a superposition with the diffeomorphism $\varphi$.

To be more specific: Let $M_{\nu}=\left\{x \in \re: x=2^{-\nu}m, m \in \mathbb{Z}^n \right\}$. Having in mind Lemma \ref{DiffRem} there is a $c_2>0$ with
\begin{align}
\label{DiffLip}
  |x-y| \leq c_2 |\varphi^{-1}(x)-\varphi^{-1}(y)|.  
\end{align}
 By a simple volume argument for $Q_{\nu,m}$ and by $|2^{-\nu}m-2^{-\nu}m'|\geq c \cdot 2^{-\nu}$ for $m\neq m'$ there is a constant $M \sim c_2^n$  such that 
\begin{align*}
  |\varphi^{-1}(M_{\nu}) \cap Q_{\nu,m}| \leq M
\end{align*}
for all $\nu \in \mathbb{N}_0,m \in \mathbb{Z}$. Hence we can take our atomic decomposition and split it into $M$ disjunct sums, i.e.
\begin{align*}
 f = \sum_{j=1}^M \sum_{\nu \in \N_0} \sum_{m \in M_{\nu,j}} \lambda_{\nu,m} a_{\nu,m} 
\end{align*}
with
\begin{align*}
 \bigcup_{j=1}^M M_{\nu,j}= \Z^n, \quad M_{\nu,j} \cap M_{\nu,j'} = \emptyset \text{ for } j\neq j'
\end{align*}
so that for all $\nu \in \N_0$, $m \in \Z^n$ and $j \in \{1,\ldots,M\}$
\begin{align}
\label{injective}
  |\left\{m' \in \Z^n: m' \in M_{\nu,j} \text{ and } \varphi^{-1}(2^{-\nu}m') \in Q_{\nu,m}\right\}|\leq 1.
\end{align}
Therefore, not more than one function $a_{\nu',m'}\circ \varphi$ is located at the cube $Q_{\nu,m}$ for each of the $M$ sums.

The support of a function $a_{\nu,m}\circ \varphi$ is contained in $\varphi^{-1}(d \cdot Q_{\nu,m})$ by \eqref{Atom1}. By Lemma \ref{DiffRem} there exists a $c_1>0$ with
\begin{align*}
 |\varphi^{-1}(x)-\varphi^{-1}(y)| \leq \frac{1}{c_1} |x-y|.
\end{align*}
Hence we get
\begin{align*}
  \varphi^{-1}(d \cdot Q_{\nu,m}) \subset c \cdot  \frac{d}{c_1} \cdot B_{2^{-\nu}}(\varphi^{-1}(2^{-\nu}m)), 
\end{align*}
where $B_r(x_0)=\left\{ x \in \re: |x-x_0|\leq r \right\}$. Hence, together with \eqref{injective} it follows: There is a constant $d'$ depending on $c_1$ such that for every $\nu \in \N_0$ and every $j \in \{1,\ldots,M\}$ there is an injective map $\Phi_{\nu,j}: M_{\nu,j} \rightarrow \Z^n$ with
\begin{align}
\label{DiffSupp}
 supp \ (a_{\nu,m}\circ \varphi) \subset d' \cdot Q_{\nu,\Phi_{\nu,j}(m)}.
\end{align}
for all $m \in M_{\nu,j}$. The constant $d'$ does not depend on $\nu$ or $m$.

Thus, if we take the derivative conditions \eqref{Atom2} and the moment conditions \eqref{Atom3} for $a_{\nu,m}\circ \varphi$ now for granted (which will be shown later), then 
\begin{align*}
  f_{j} \circ \varphi=\sum_{\nu \in \N_0} \sum_{m \in M_{\nu,j}} \lambda_{\nu,m} (a_{\nu,m} \circ \varphi)
\end{align*}
is an atomic decomposition of the function $f_j \circ \varphi$. Finally, we have to look at the sequence space norms, see Definition \ref{DefSeq}. 

We will concentrate on the $\fspq$-case since the $\bspq$-case is easier because it does not matter if one changes the order of summation over $m$. By the atomic representation theorem and \eqref{DiffSupp} we will have
\begin{align*}
 \|f_j \circ \varphi|\fspq\| \leq c \left\|\left(\sum_{\nu=0}^{\infty} \sum_{m \in M_{\nu,j}} |\lambda_{\nu,m}\chi_{\nu,\Phi_{\nu,j}(m)}^{(p)}(\cdot)|^q\right)^{\frac{1}{q}}\big|L_p(\re)\right\|.
\end{align*}
To transfer this into the usual sequence space norm we make use of
\begin{align}
\label{MapLip}
  Q_{\nu,\Phi_{\nu,j}(m)} \subset \varphi^{-1} (c \cdot Q_{\nu,m}) 
\end{align}
with a constant $c$ depending on $c_2$ from \eqref{DiffLip}, but independent of $\nu$ and $m$. This follows from $\varphi^{-1} (2^{-\nu}m) \in Q_{\nu,\Phi_{\nu,j}(m)}$. Hence assuming that $a_{\nu,m}\circ \varphi$ fulfil \eqref{Atom2} and \eqref{Atom3} we obtain
\begin{align}
\label{DiffEst}
\begin{split}
 \|f_{j} \circ \varphi|\fspq\| &\leq c \left\|\left(\sum_{\nu=0}^{\infty} \sum_{m \in M_{\nu,j}} |\lambda_{\nu,m}\chi_{\nu,m}^{(p)}(\varphi (\cdot))|^q\right)^{\frac{1}{q}}\big|L_p(\re)\right\| \\
	  &\leq c' \left\|\left(\sum_{\nu=0}^{\infty} \sum_{m \in M_{\nu,j}} |\lambda_{\nu,m}\chi_{\nu,m}^{(p)}(\cdot)|^q\right)^{\frac{1}{q}}\big|L_p(\re)\right\| \\
&\leq c' \left\|\left(\sum_{\nu=0}^{\infty} \sum_{m \in \Z^n} |\lambda_{\nu,m}\chi_{\nu,m}^{(p)}(\cdot)|^q\right)^{\frac{1}{q}}\big|L_p(\re)\right\| \\
	  &\leq c'' \|f|\fspq\|.
\end{split}
\end{align} 
In the first step we used \eqref{MapLip}, in the second step we used Lemma \ref{Lpdiff} and part (iii) of Lemma \ref{DiffRem} and in the last step we applied the atomic decomposition theorem for $f$. As done in the first step, one can replace the characteristic function of $c \cdot Q_{\nu,m}$ by the characteristic function of $Q_{\nu,m}$ in the sequence space norm getting equivalent norms, see \cite[section 1.5.3]{Tri08}. This can be proven using the Hardy-Littlewood maximal function.

Finally, we have to take a look at the derivative conditions \eqref{Atom2} and the moment conditions \eqref{Atom3}. The latter part is also considered in Lemma 5 of \cite{Skr98} using the atomic approach with condition \eqref{Atom21}. 

Let $a_{\nu,m}$ be an $(s,p)_{K,L}$-atom and let $\rho\geq \max(K,L+1)$. If we can show that $\varphi \circ a_{\nu,m}$ is an $(s,p)_{K,L}$-atom as well, we are done with the proof since we can choose $K$ and $L$ suitably small enough by the atomic decomposition theorem \ref{AtomicRepr}. Let $T_{\nu}(x):=2^{-\nu}x$ and ${\cal T}_{\nu}(\varphi)= T_{\nu}^{-1} \circ \varphi \circ T_{\nu}$. Then
\begin{align*}
 \|\left(a_{\nu,m}\circ \varphi\right)(2^{-\nu}\cdot)|\hold[K]\|&=  \|a_{\nu,m}\circ \varphi \circ T_{\nu}|\hold[K]\| = \|a_{\nu,m}\circ T_{\nu} \circ {\cal T}_{\nu}(\varphi)|\hold[K]\|.
\end{align*}
By a simple dilation argument for the Hölder spaces $\hold[\rho-1]$ it holds 
\begin{align*}
  \left\|\frac{\partial \left({\cal T}_{\nu}(\varphi)\right)_i}{\partial x_j}\big|\hold[\rho-1]\right\|\leq \left\|\frac{\partial \varphi_i}{\partial x_j}\big|\hold[\rho-1]\right\|
\end{align*}
for all $i,j\in \{1,\ldots,n\}$ and $\nu \in \N_0$. Hence by Lemma \ref{HoelderDiff} and Remark \ref{DiffUni} we find a constant $C$ independent of $\nu$ and $m$ such that
\begin{align*}
  \|\left(a_{\nu,m}\circ \varphi\right)(2^{-\nu}\cdot)|\hold[K]\|= \|a_{\nu,m}\circ T_{\nu} \circ {\cal T}_{\nu}(\varphi)|\hold[K]\| \leq C \cdot 
 \|a_{\nu,m}(2^{-\nu}\cdot)|\hold[K]\|
\end{align*}
So the derivative condition \eqref{Atom2} is shown. 

Regarding the moment condition \eqref{Atom3} of $a_{\nu,m}\circ \varphi$ we consider two cases: At first, let $\varphi$ be a $\rho$-diffeomorphism with $\rho>1$. Then $\varphi$ and $\varphi^{-1}$ are differentiable. We use the moment condition of $a_{\nu,m}$ itself and Lemma \ref{HoelderDiff} to get
\begin{align*}
\left| \, \int\limits_{d' \cdot Q_{\nu,\Phi_{\nu,j}(m)}} \psi(x) \cdot a(\varphi(x)) \ dx \right| &=\left|\,\int\limits_{\varphi^{-1}\left(d \cdot Q_{\nu,m}\right)} \psi(x) \cdot a(\varphi(x)) \ dx \right| \\
&=\left| \, \int\limits_{d \cdot Q_{\nu,m}} \psi\left(\varphi^{-1}(x)\right) \cdot |\det \varphi^{-1}|(x) \cdot a(x) \ dx \right| \\
&\leq C  \cdot 2^{-\nu\Kap} \cdot \||\det \varphi^{-1}(x)|\cdot \left(\psi\circ \varphi^{-1}\right)|\hold[L]\| \\
&\leq C' \cdot 2^{-\nu\Kap} \cdot \|\psi|\hold[L]\|.
\end{align*}
We used the transformation formula for integrals and 
\begin{align*}
  \det\ J \left(\varphi^{-1}\right)\in \hold[L]
\end{align*}
since $\varphi$ is a $\rho$-diffeomorphism with $\rho \geq L+1$. Furthermore, the sign of $\det J\left(\varphi^{-1}\right)$ is constant. 

If $\rho=1$, then $L=0$ by our choice of $\rho$. This means, that no moment conditions are needed. Hence we have nothing to prove. The choice of $\rho=1$ is only allowed if $\sigma_p<s<1$ resp. $\sigma_{p,q}<s<1$. 

For some further technicalities similar as in Remark \ref{ProdDef} see Remark \ref{DiffDef}.
\end{bew}

\begin{Bemerkung}
 This has been proven (in a sketchy way) in Lemma 3 in \cite{Skr98} for the more special atomic definition there. 
\end{Bemerkung}

\begin{Bemerkung}
 If $\sigma_p<s<1$ resp. $\sigma_{p,q}<s<1$, then the choice of $\rho=1$ is possible for these values of $s$. This gives the same result as in Proposition 4.1 in \cite{Tri02}, where the notation of Lipschitz diffeomorphisms as in Definition \ref{LipDef} is used. This results in
\begin{Theorem}
Let $0<q\leq \infty$.

(i) Let $0<p\leq \infty$ and $\sigma_p<s<1$. If $\varphi:\re \rightarrow \re$ is a bi-Lipschitzian map, then there exists a constant $c$ such that
\begin{align*}
 \|f(\varphi(\cdot))|\bspq\| \leq c \cdot \|f|\bspq\|.
\end{align*}
for all $f \in \bspq$. Hence $D_{\varphi}$ maps $\bspq$ onto $\bspq$.

(ii) Let $0<p<\infty$ and $\sigma_{p,q}<s<1$. If $\varphi:\re \rightarrow \re$ is a bi-Lipschitzian map, then there exists a constant $c$ such that
\begin{align*}
 \|f(\varphi(\cdot))|\fspq\| \leq c \cdot \|f|\fspq\|.
\end{align*}
for all $f \in \fspq$. Hence $D_{\varphi}$ maps $\fspq$ onto $\fspq$.
\end{Theorem}

\end{Bemerkung}

\begin{Bemerkung}
\label{DiffDef} 
We have to deal with some technicalities of the proof of Theorem \ref{Diffeo}. We concentrate on the $\bspq$-case, the $\fspq$-case is nearly the same.

Let at first be $\rho>1$. In principle, Theorem \ref{Diffeo} and Lemma \ref{HarmS'-KonvAtom} show that
\begin{align}
\label{AtomKonv2}
 \sum_{\nu} \sum_{m}\lambda_{\nu,m} (a_{\nu,m} \circ \varphi) 
\end{align}
converges unconditionally in ${\cal S}'(\re)$, where
\begin{align*}
 f=\sum_{\nu} \sum_{m}\lambda_{\nu,m} a_{\nu,m} \text{ in } {\cal S}'(\re),
\end{align*}
and the limit belongs to $\bspq$ if $f$ belongs to $\bspq$. 

To define the superposition of $f$ and $\varphi$ as this limit, we have to show that the limit does not depend on the atomic decomposition we chose for $f$. Let $\psi \in C^{\infty}(\re)$ with compact support be given. Then
\begin{align*}
 \sum_{\nu} \sum_{m}\left| \int_{\re}  \lambda_{\nu,m} \left(a_{\nu,m}\circ \varphi \right)(x) \psi(x) \ dx\right|
 =\sum_{\nu} \sum_{m}\left| \int_{\re}  \lambda_{\nu,m} a_{\nu,m}(x) \left[ \psi\left(\varphi^{-1}(x)\right) \cdot |\det \varphi^{-1}(x)| \right] \ dx\right|
\end{align*}
makes sense, see \eqref{KonvB}, because by Lemma \ref{HoelderDiff} the function $\psi\left(\varphi^{-1}(x)\right) \cdot |\det \varphi^{-1}(x)|$ has compact support and belongs to $\hold[M]$ for a suitable $M>0$ with $M>\sigma_p-s$. Now the achievements at the end of Corollary \ref{FolgLocal} show that this integral limit does not depend on the choice of the atomic decomposition for $f$. Hence we obtain that the limit in \eqref{AtomKonv2} (considered as an element in ${\cal S}'(\re)$) is the same for all choices of atomic decompositions.  

If the choice of $\rho=1$ is allowed, then automatically $s>\sigma_p$ and $\bspq$ consists of regular distributions by Sobolev's embedding. Hence the superposition of $f \in \bspq \subset L_p(\re)$ for $1 \leq p \leq \infty$ resp. $f \in \bspq \subset L_1(\re)$ for $0<p\leq 1$  with a $1$-diffeomorphism $\varphi$ is defined as the superposition of a regular distribution with a $1$-diffeomorphism and is continuous as an operator from $L_p(\re)$ resp. $L_1(\re)$ to $L_p(\re)$ resp. $L_1(\re)$ by Lemma \ref{Lpdiff}.

If $p<\infty$, then atomic decompositions of $f \in \bspq$ converge to $f$ with respect to the norm of $L_p(\re)$ for $1\leq p < \infty$ resp. with respect to the norm of $L_1(\re)$ for $0<p<1$, see \cite[Section 2.12]{Tri06}. Hence the limit does not depend on the choice of the atomic decomposition and is equal to the usual definition of the superposition of a regular distribution $f$ and the $1$-diffeomorphism $\varphi$. 

If $p=\infty$, we use the local convergence of the atomic decompositions of $f$ in $L_{\infty}(\re)$, i.e. we restrict $f$ and its atomic decomposition to a compact subset $K$ of $\re$. Then this restricted atomic decomposition converges to the restricted $f$ with respect to the norm of $L_{\infty}(K)$. This suffices to prove uniqueness of the limit which is an $L_{\infty}(\re)$-function.
\end{Bemerkung}

\begin{Bemerkung}
For fixed $s,p$ and $q$ the constant $c$ in Theorem \ref{Diffeo} depends on the $\rho$-diffeomorphism $\varphi$. Looking into the proof of Theorem \ref{Diffeo} and Remark \ref{DiffUni} the following definition is useful: 
\end{Bemerkung}

\begin{Definition}
 Let $\rho\geq 1$. We call $\{\varphi^m\}_{m \in \N}$ a bounded sequence of $\rho$-diffeomorphisms if every $\varphi^m$ is a $\rho$-diffeomorphism, if there are universal constants $c_1,c_2>0$ with
\begin{align*}
 c_1\leq \frac{|\varphi^m(x)-\varphi^m(y)|}{|x-y|} \leq c_2 
\end{align*}
for $m \in \N$, $x,y\in \re$ with $0<|x-y|\leq 1$ and if - for $\rho>1$ - there is a universal constant $c$ with   
\begin{align*}
\sum_{i=1}^n\sum_{j=1}^n  \left\|\frac{\partial \varphi_i^m}{\partial x_j}\big|\hold[\rho-1]\right\|< c.
\end{align*}
for $m \in \N$.
\end{Definition}

\begin{Bemerkung}
 If $\{\varphi^m\}_{m \in \N}$ is a bounded sequence of $\rho$-diffeomorphisms, then $(\varphi^m)^{-1}$ exists for all $m \in \N$ and $\{(\varphi^m)^{-1} \}_{m \in \N}$ is a bounded sequence of $\rho$-diffeomorphisms, too. This follows by the arguments of Lemma \ref{DiffRem}.
\end{Bemerkung}
Now, by going through the proof of Theorem \ref{Diffeo} and Remark \ref{DiffUni} it follows
\begin{Folgerung}
 Let $s \in \mathbb{R}$, $0<q\leq \infty$ and $\rho\geq 1$.

(i) Let $0<p\leq \infty$ and $\rho > \max(s,1+\sigma_p-s)$. If $\{\varphi^m\}_{m \in \N}$ is a bounded sequence of $\rho$-diffeomorphisms, then there exists a constant $C$ such that
\begin{align*}
 \|f(\varphi^m(\cdot))|\bspq\| \leq c \cdot \|f|\bspq\|
\end{align*}
for all $f \in \bspq$ and $m \in \N$.

(ii) Let $0<p<\infty$ and $\rho > \max(s,1+\sigma_{p,q}-s)$. If $\{\varphi^m\}_{m \in \N}$ is a bounded sequence of $\rho$-diffeomorphisms, then there exists a constant $C$ such that
\begin{align*}
 \|f(\varphi^m(\cdot))|\fspq\| \leq c \cdot \|f|\fspq\|
\end{align*}
for all $f \in \fspq$ and $m\in \N$.
\end{Folgerung}


\bibliographystyle{abbrv}
\bibliography{ben}

\end{document}